\documentclass[a4paper,11pt,oneside]{article}
\pdfoutput=1
\RequirePackage{amsmath, amsthm, amssymb, mathtools, mathrsfs}
\RequirePackage{natbib}
\RequirePackage{tikz, tikz-cd}
\RequirePackage{tabularray}
\RequirePackage{xspace}

\RequirePackage{mdframed}

\RequirePackage[margin=0.9in]{geometry}
\RequirePackage{libertine}
\RequirePackage{courier}
\RequirePackage{mathpazo}
\RequirePackage{xcolor}
\RequirePackage{pifont}
\RequirePackage{authblk}
\RequirePackage[backref=page]{hyperref}
\RequirePackage{enumerate}
\RequirePackage{titlesec}

\definecolor{wred}{rgb}{0.533333,0.10980,0.15686}
\definecolor{wredlight}{rgb}{0.788, 0.11, 0.157}
\titleformat{\section}
{\color{wred}\normalfont\Large\bfseries}
{\color{wred}\thesection}{1em}{}
\titleformat{\subsection}
{\color{wred}\normalfont\Large\bfseries}
{\color{wred}\thesubsection}{1em}{}

\hypersetup{
	colorlinks=true,
  linkcolor=wredlight,
  citecolor=wredlight,
  filecolor=wredlight,
  urlcolor=wredlight
}

 \renewcommand*{\backrefalt}[4]{%
    \ifcase #1%
     \or (page:~#2)%
     \else (pages:~#2)%
    \fi%
    }

    \makeatletter
\def\@fnsymbol#1{\ensuremath{\ifcase#1\or \dagger\or \ddagger\or
   \mathsection\or \mathparagraph\or \|\or **\or \dagger\dagger
   \or \ddagger\ddagger \else\@ctrerr\fi}}
    \makeatother

\newtheorem{theorem}{Theorem}[section]
\newtheorem{lemma}{Lemma}[section]
\newtheorem{definition}{Definition}[section]
\newtheorem{proposition}{Proposition}[section]

\newtheorem{example}{Example}[section]
\newtheorem{corollary}{Corollary}[section]
\newtheorem{remark}{Remark}[section]

\providecommand{\keywords}[1]
{
  \small	
  \textbf{\textit{Keywords---}} #1
}
\DeclareMathOperator{\diag}{diag}
\DeclareMathOperator*{\argmin}{arg\,min}

\DeclareMathOperator{\aff}{aff}
\DeclareMathOperator{\spann}{span}
\DeclareMathOperator{\ehull}{e-hull}

\newcommand{\hadamard}{\odot}
\newcommand{\zero}{0}
\newcommand{\plus}{+}
\newcommand{\set}[1]{\left\{ #1 \right\}}
\newcommand{\eqdef}{\triangleq}
\newcommand{\trn}{^\intercal}
\newcommand{\abs}[1]{\left| #1 \right|}
\newcommand{\nrm}[1]{\left\Vert #1 \right\Vert}

\newcommand{\pred}[1]{\delta\left[#1\right]}
\newcommand{\PR}[2][]{\mathbb{P}_{#1}\left( #2 \right)}

\newcommand{\eps}{\varepsilon}
\newcommand{\bigO}{\mathcal{O}}
\newcommand{\kl}[2]{D\left(#1 \middle| \middle| #2\right)}
\newcommand{\stoch}{{\mathfrak{s}}}
\newcommand{\bbR}{\mathbb{R}}
\newcommand{\bbN}{\mathbb{N}}

\newcommand{\calN}{\mathcal{N}}
\newcommand{\calR}{\mathcal{R}}
\newcommand{\calS}{\mathcal{S}}
\newcommand{\calT}{\mathcal{T}}
\newcommand{\calB}{\mathcal{B}}
\newcommand{\calX}{\mathcal{X}}
\newcommand{\calY}{\mathcal{Y}}

\newcommand{\calE}{\mathcal{E}}
\newcommand{\calM}{\mathcal{M}}
\newcommand{\calW}{\mathcal{W}}

\newcommand{\calH}{\mathcal{H}}

\newcommand{\calU}{\mathcal{U}}
\newcommand{\calV}{\mathcal{V}}
\newcommand{\calF}{\mathcal{F}}
\newcommand{\calL}{\mathcal{L}}
\newcommand{\calJ}{\mathcal{J}}
\newcommand{\calD}{\mathcal{D}}
\newcommand{\calG}{\mathcal{G}}

\newcommand{\enumi}{(i)\xspace}
\newcommand{\enumii}{(ii)\xspace}
\newcommand{\enumiii}{(iii)\xspace}
\newcommand{\enumiv}{(iv)\xspace}
\newcommand{\enumv}{(v)\xspace}

\title{\vspace{-1.0cm} 
Characterization of Exponential Families of \\ Lumpable Stochastic Matrices}
\begin{document}

\author[1]{Shun Watanabe\thanks{email: shunwata@cc.tuat.ac.jp.}}
\author[1]{Geoffrey Wolfer\thanks{email: wolfer@go.tuat.ac.jp. 
}}
\affil[1]{Tokyo University of Agriculture and Technology, Tokyo}

\date{\today}

\maketitle

\begin{abstract}
    It is known that the set of lumpable Markov chains over a finite state space, with respect to a fixed lumping function, generally does not form an exponential family of stochastic matrices. In this work, we explore efficiently verifiable necessary and sufficient conditions for families of lumpable transition matrices to form exponential families. To this end, we develop a broadly applicable dimension-based method for determining whether a given family of stochastic matrices forms an exponential family.
\end{abstract}


\tableofcontents

\section{Introduction}
\label{section:introduction}
Exponential families (e-families) of distributions are of established importance in statistics due to their distinctive properties for inference problems. For instance, they uniquely provide sufficient statistics capable of condensing any amount of independent and identically distributed data into a fixed number of values \citep{pitman1936sufficient, koopman1936distributions, darmois1935lois}. What is more, it is known that the maximum likelihood estimator achieves the Cram\'{e}r-Rao lower bound only\footnote{Note that this fact only holds when imposing additional regularity conditions on the family.} when the family of distributions forms an exponential family \citep{wijsman1973attainment,joshi1976attainment,fabian1977cramer, muller1989attainment}.
In the language of information geometry \citep{amari2007methods}, 
positive probability distributions are endowed with the structure of a smooth manifold with a pair of dual affine connections---the e-connection and m-connection---and statistical models are regarded as submanifolds. 
In this framework, being an e-family geometrically corresponds to being autoparallel with respect to the e-connection.
Furthermore, deviation from being an e-family---and the subsequent breakdown of the statistical properties---can be measured in terms of curvature; this characterizes second order efficiency of estimators \citep{efron1975defining}.
Recently, e-families have also been put under the spotlight in optimization since they allow for efficient natural-gradient computation \citep{amari1998natural}, which finds application in machine learning.

It is possible to similarly construct a dually flat geometry on the space of irreducible stochastic matrices \citep{nagaoka2005exponential} defined over a fixed strongly connected transition digraph. Independently and identically distributed (iid) processes, which can be regarded as memoryless Markov chains, are known to form an e-family in the larger family of irreducible stochastic matrices \citep{ito1988}.
This Markovian framework is consistent with the divergence rate of the corresponding Markov processes, and
information projections \citep{boza1971asymptotically, csiszar1987conditional}
which arise naturally from the study of large deviations \citep{moulos2019optimal} and hypothesis testing \citep{nakagawa1993converse, watanabe2017}.
In this regard, it strictly encompasses the vanilla framework for distributions, while accommodating for processes with time dependencies.

However, Markov processes
exhibit a significantly richer structure than their iid counterparts, with numerous properties---such as irreducibility, aperiodicity or time-reversibility---that are not pertinent to iid processes but are of well-established importance for Markov chains. A recently initiated research program seeks to analyze how these 
Markov-centric properties translate into information geometric features of the corresponding families of
stochastic matrices.
For instance, a Markov chain having a uniform stationary distribution is equivalent to being represented by a doubly stochastic transition matrix and it is well-established that the set of doubly stochastic matrices forms a mixture family \citep{hayashi2014information}.
Similarly, verifying the detailed-balance equation---indicating the time-reversibility of the stochastic process---means that the transition matrix is self-adjoint in a certain Hilbert space and the set of reversible stochastic matrices is known to form both a mixture family and an exponential family \citep{wolfer2021information}.
For context trees, it is known that a tree model forms an e-family if and only if it is an FSMX model
\citep{takeuchi2007exponential, takeuchi2017information}.
More recently, \citet{wolfer2024geometric} began analyzing lumpability of Markov chains \citep{kemeny1983finite}.
Lumpable Markov chains allow for the reduction of the state space by merging symbols without losing Markovianity, making them highly practical.
Specifically, the authors showed that although the lumpable set with respect to a fixed lumping map typically forms neither a mixture family nor an e-family, it is still possible to endow the family with the structure of a mutually dual foliated manifold, leading to a mixed coordinate system \citep[Chapter~3.7]{amari2007methods}.
Their construction is centered around the concept of a Markov embedding, defined as a right inverse of the lumping operation, which they argue serves a similar role to \v{C}encov's statistical morphisms \citep{cencov1983statistical} in the context of Markov chains.

The problem of selecting a good statistical model involves choosing one that has favorable analytical properties. In this regard, both e-families and lumpable families are highly sought-after models, and practitioners may be interested in enjoying the best of both worlds.
However, as previously mentioned, lumpable families do not generally form e-families. Indeed, they may or may not be e-families depending on their connection graph and the lumping map. This phenomenon contrasts with many previously analyzed classes; for instance, the set of reversible stochastic matrices, forms an e-family for any symmetric connection graph.
In this paper, we initiate the problem of characterizing the conditions under which lumpable stochastic matrices do form e-families.
Since Markov embeddings
demonstrably preserve e-families of stochastic matrices, they naturally generate one class of lumpable e-families.
However, this approach proves to be quite restrictive.
Perhaps surprisingly, it is possible to construct families that are not directly derived from the embedding of an e-family.
In this work we explore some necessary and sufficient conditions for the lumpable set to form an e-family.

\subsection*{Major contributions}
We summarize our main results below.
\begin{enumerate}[$\diamond$]
    \item \textbf{Necessary and sufficient criteria based on combinatorial properties of the connection digraph.} We obtain conditions on the lumpable family $\calW_{\kappa}(\calY, \calE)$ for being an e-family in terms of so-called multi-row merging blocks (refer to Definition~\ref{definition:merging-block}). 
    Namely, a sufficient condition (Corollary~\ref{corollary:no-multi-row-merging-block-is-sufficient}) for $\calW_\kappa(\calY, \calE)$ to be an e-family is that it exhibits no multi-row merging block, while if it exhibits a multi-row merging block which is redundant (Definition~\ref{definition:redundant-block}), this precludes the lumpable family from being exponential (Theorem~\ref{theorem:redundant-merging-block-criterion}). However, neither of the above conditions fully characterizes the property of being an e-family.
    In fact, we provide an alternative sufficient criteria (Proposition~\ref{proposition:lazy-cycle-criterion}), which shows in particular that an e-family could have an arbitrarily large number of multi-row merging blocks.
\item
\textbf{Complete characterization based on a novel dimensional criterion.}
Our chief technical contribution is a new methodology for determining whether a given family of stochastic matrices forms an exponential family, based on a dimension argument (refer to Theorem~\ref{theorem:general-dimensional-criterion}). 
Informally, for any subfamily $\calV$ of $\calW(\calY, \calE)$ which is defined by a finite set of linear constraints, $\calV$ forms an e-family if and only if $\dim \calV = \dim \ehull(\calV)$, where $\ehull(\calV)$ is the minimal e-family (refer to Definition~\ref{definition:e-hull}) generated by $\calV$. In this widely applicable setting, the problem thus reduces to the computation of the dimensions of $\calV$ and its e-hull.

Our new method contrasts with the more prevalent method of analyzing geodesics---which is for instance used to prove our redundant block criterion---in at least two aspects. First, while one can in principle rule out exponential families by exhibiting a geodesic and proving it leaves the manifold, constructing such an analytically tractable witness presents a challenge. 
Second, while one can confirm the e-family property heuristically by considering a large number of geodesics, it is practically impossible to inspect all geodesics.
We argue that our approach is also preferable to computing the e-curvature, which is often hard.

In particular, we demonstrate the practicality of our technique for determining whether the lumpable family is exponential.
We show that the e-hull of lumpable stochastic matrices can be expressed as the sum of the affine hull of the logarithm of positive lumpable functions and a linear space of functions (refer to Definition~\ref{definition:anti-shift-functions}).
As a consequence, to determine whether the lumpable family is exponential, it suffices to compare its manifold dimension to the dimension of the e-hull computed above  (refer to Theorem~\ref{theorem:dimensional-criterion}).
However, while the manifold dimension of the lumpable family is known \citep{wolfer2024geometric}, the dimension of its e-hull is not.
We solve this problem by constructing a basis for the e-hull of the positive lumpable cone 
and we show that as a result, the dimension of the e-hull can be computed using a polynomial-time algorithm.

We then specialize the above result into Corollary~\ref{corollary:dimensional-criterion-simplified} to obtain a necessary condition purely based on more elementary combinatorial properties of the connection graph and the lumping map.

\item \textbf{Monotonicity and stability.}
Finally, we examine how the property of being an e-family changes through basic operations on the edge set $\calE$.
In particular, we exhibit a monotonicity property of e-families of lumpable stochastic matrices. Namely, 
we show in Theorem~\ref{theorem:monotonicity} that it is generally the case that when $\calE \subset \calE'$,
if $\calW_\kappa(\calY, \calE')$ forms an e-family, then $\calW_\kappa(\calY, \calE)$ also forms an e-family.
\end{enumerate}

\begin{table}[]
\centering
    \caption{Table of Notation}
    \begin{tabular}{l|l|l}
   Symbol & Meaning & Definition \\
        $\calX$ & state space (typically smaller than $\calY$) & p.\pageref{nom:smaller-state-space} \\
        $\calY$ & state space (typically larger than $\calX$) & p.\pageref{nom:larger-state-space} \\
        $\calD$ & set of edges such that the graph $(\calX, \calD)$ is strongly connected & p.\pageref{nom:smaller-edge-set} \\
        $\calE$ & set of edges such that the graph $(\calY, \calE)$ is strongly connected & p.\pageref{nom:larger-edge-set} \\
        $\kappa : \calY \to \calX$ & surjective lumping map & p.\pageref{nom:lumping-map}\\ 
        $\stoch$ & normalization map from irreducible matrix to stochastic matrix & p.\pageref{nom:stochastic-rescaling} \\
        $\calS_x$ & preimage of $x \in \calX$ by $\kappa$ & p.\pageref{nom:x-preimage-kappa} \\
        $\calF(\calY, \calE)$ & real functions over $\calE$ & p.\pageref{nom:real-functions} \\
        $\calF^+(\calY, \calE)$ & positive functions over $\calE$ & p.\pageref{nom:positive-real-functions} \\
        $\calW(\calY, \calE)$ & stochastic matrices with connection graph $(\calY, \calE)$  & p.\pageref{nom:irreducible-stochastic-matrices} \\
        $\calW_\kappa(\calY, \calE)$ & $\kappa$-lumpable stochastic matrices with connection graph $(\calY, \calE)$ & p.\pageref{nom:lumpable-stochastic-matrices} \\
        $\calF_\kappa(\calY, \calE)$ & $\kappa$-lumpable real functions over $\calE$ & p.\pageref{nom:lumpable-functions} \\
        $\calF^+_\kappa(\calY, \calE)$ & $\kappa$-lumpable positive functions over $\calE$ & p.\pageref{nom:positive-lumpable-functions} \\
        $\calN(\calY, \calE)$ & functions in $\calF(\calY, \calE)$ that can be expressed as $f(y') - f(y) + c$,  &  \\
         & where $f \colon \calY \to \bbR$ and $c$ is a constant & p.\pageref{nom:antishift-plus-constant} \\
        $\calG(\calY, \calE)$ & quotient space $\calF(\calY, \calE) / \calN(\calY, \calE)$ & p.\pageref{nom:quotient-space} \\
        $\odot$ & Hadamard product & p.\pageref{nom:hadamard-product} \\
        $\ehull$ & exponential hull of a family of stochastic matrices & p.\pageref{nom:e-hull}
    \end{tabular}
    
    \label{table:nomenclature}
\end{table}

\section{Preliminaries}
\label{section:preliminaries}
\subsection{Notation}

We let $(\calY, \calE)$ be a directed graph (digraph) with finite vertex set \label{nom:larger-state-space}$\calY$ and edge set \label{nom:larger-edge-set}$\calE \subset \calY^2$. We assume that
$(\calY, \calE)$ is strongly connected, that is every vertex is reachable from every other vertex by traversing edges in their proper direction.
For $\{Y_t\}_{t \in \bbN}$ a time-homogeneous discrete time Markov chain (DTMC) over the space space $\calY$, we collect the transition probabilities into a row-stochastic matrix $P$. In other words, we write\footnote{
Our notation follows
the applied probability literature.
In the information theory literature, $P(y'|y)$ is sometimes used in lieu of $P(y, y')$.}
$$P(y,y') = \PR{Y_{t+1} = y' | Y_{t} = y}.$$
When $P(y,y') > 0$ if and only if $(y,y') \in \calE$, we say that $(\calY, \calE)$ is a connection graph for $P$.
We denote \label{nom:irreducible-stochastic-matrices}$\calW(\calY, \calE)$ the set of all irreducible row-stochastic matrices pertaining to the connection graph $(\calY, \calE)$.
We additionally define \label{nom:real-functions}$\calF(\calY, \calE) = \bbR^{\calE}$ the set of all real functions on the set of edges $\calE$ and 
\label{nom:positive-real-functions}$\calF^+(\calY, \calE)$ its positive subset. As it allows us to conveniently write a function $F \in \calF(\calY, \calE)$ in the form of a square matrix, we will routinely identify
\begin{equation*}
\begin{split}
    \calF(\calY, \calE) &\cong \set{ F \in \bbR^{\calY^2} \colon \forall (y,y') \not \in \calE \implies F(y,y') = 0 }, \\
    \calF^+(\calY, \calE) &\cong \set{ F \in \calF(\calY, \calE) \colon \forall (y,y') \in \calE \implies F(y,y') > 0 }. \\
\end{split}
\end{equation*}
Note that via the above identification, one can write $\calW(\calY, \calE) \subset \calF^+(\calY, \calE)$.
The Hadamard product of $A$ and $B$ in $\calF(\calY, \calE)$ is denoted \label{nom:hadamard-product}$A \odot B$ and for $t \in \bbR$, $A^{\odot t}$ is defined as the function such that for any $y,y' \in \calE$, $A^{\odot t}(y,y') = A(y,y')^t$.
We overload $\exp$ and $\log$ as follows,
\begin{equation*}
\begin{split}
    \exp \colon \calF(\calY, \calE) &\to \calF^+(\calY, \calE), \\
    \log\colon \calF^+(\calY, \calE) &\to \calF(\calY, \calE),
\end{split}
\end{equation*}
where for any $F \in \calF(\calY, \calE)$ and $(y,y') \in \calE$, $\exp(F)(y,y') = \exp(F(y,y'))$, and for any $F \in \calF^+(\calY, \calE)$ and $(y,y') \in \calE$, $\log(F)(y,y') = \log(F(y,y'))$.

\subsection{Lumpability}
\label{section:lumpability}

One classical operation on Markov processes is lumping \citep{kemeny1983finite}, which means merging symbols together and recording the observations on the reduced space. It is well known that this operation typically disrupts the Markov property \citep{burke1958markovian, rogers1981markov}. Chains for which the Markov property is preserved are called lumpable. More formally, for another state space \label{nom:smaller-state-space}$\calX$ with $\abs{\calX} \leq \abs{\calY}$, and for a surjective symbol merging map \label{nom:lumping-map}$\kappa \colon \calY \to \calX$, we say that the Markov chain
 $\{Y\}_{t \in \bbN}$ with transition matrix
 $P \in \calW(\calY, \calE)$ is $\kappa$-lumpable whenever the stochastic process $\{\kappa(Y_t)\}_{t \in \bbN}$ also forms a time-homogeneous DTMC. In particular, introducing the lumped edge set
\begin{equation*}
    \label{nom:smaller-edge-set}\calD \eqdef \kappa(\calE) \eqdef \set{ (\kappa(y), \kappa(y')) \colon (y,y') \in \calE } \subset \calX^2,
\end{equation*}
we note that the graph $(\calX, \calD)$ is strongly connected, and that the transition matrix of the lumped process, denoted $P^\flat$, satisfies $P^\flat \in \calW(\calX, \calD)$, where $\calW(\calX, \calD)$ is defined similarly to $\calW(\calY, \calE)$.
Denoting \label{nom:lumpable-stochastic-matrices}$\calW_\kappa(\calY, \calE)$ the $\kappa$-lumpable subset of $\calW(\calY, \calE)$,
observe that $\kappa$ induces 
a push-forward $\kappa_\star$ on stochastic matrices,
\begin{equation}
\begin{split}
\label{eq:push-forward-lumping}
    \kappa_\star \colon \calW_\kappa(\calY, \calE)  \to \calW(\calX, \calD)
\end{split}
\end{equation}
as well as a partition of the space $\calY$, which we denote by
\begin{equation*}
    \calY = \biguplus_{x \in \calX} \calS_x,
\end{equation*}
where for any $x \in \calX$, we wrote \label{nom:x-preimage-kappa}$\calS_x \eqdef \kappa^{-1}(x)$.
The following characterization of $\calW_\kappa(\calY, \calE)$ was provided by \citet{kemeny1983finite}. It holds that $P \in \calW_{\kappa}(\calY, \calE)$
if and only if for any $(x,x') \in \calD$ and any $y_1, y_2 \in \calS_{x}$,
\begin{equation}
\label{eq:kemeny-snell-condition}
    P(y_1, \calS_{x'}) = P(y_2, \calS_{x'}),
\end{equation}
where we used the shorthand $P(y, \calS_{x'}) = \sum_{y' \in \calS_{x'}} P(y,y')$ for $(x,x') \in \calD$ and $y \in \calS_{x}$.
It is instructive to note that
the Kemeny--Snell condition enables us
to provide an explicit form of $\kappa_\star P$. Indeed, for any $(x,x') \in \calD$, it holds that $\kappa_\star P(x,x') = P(y, {\calS}_{x'})$ for any $y \in \calS_{x}$.
We will use the notation\footnote{We note that in differential geometry $\sharp$ and $\flat$ commonly denote the musical isomorphism. However, in this paper, we use these symbols differently.} of \citet{levin2009markov} to often---albeit not always---disambiguate objects which pertain to the larger space using the superscript $\sharp$ and which pertain to the reduced space using the superscript $\flat$.
Note that the Kemeny--Snell condition \eqref{eq:kemeny-snell-condition}  
can be rewritten in matrix form. Indeed,
for any $P \in \calW(\calY, \calE)$,
    from \citep{barr1977eigenvector},
    we have that $P \in \calW_\kappa(\calY, \calE)$ if and only if,
\begin{equation*}
    K \overline{K} \trn P K = P K,
\end{equation*}
where $K$ is a $\abs{\calY} \times \abs{\calX}$ matrix of entries in $\mathbb{B} = \set{0,1}$ identified with the map
\begin{equation*}
\begin{split}
    K \colon \calY \times \calX &\to \set{0,1} \\
    (y,x) &\mapsto K(y,x) = \begin{cases}
                1 &\text{ when } y \in \calS_x \\
                0 &\text{ otherwise,}
            \end{cases}
\end{split}
\end{equation*}
and \begin{equation*}
    \overline{K}\trn = (K \trn K)^{-1} K \trn
\end{equation*}
is a rescaling of $K \trn$ such that rows are normalized to probability vectors.
As a result, the $\kappa$-lumping condition for $P$ can be rewritten as
\begin{equation}
\label{eq:kemeny-snell-condition-matrix-form}
    (K \overline{K}\trn - I )P K = 0.
\end{equation}
Two lumpable families of stochastic matrices are said to be equivalent if they coincide upon relabeling of the state space and lumped state space.
Namely, for $\kappa_1 \colon \calY_1 \to \calX_1, \kappa_2 \colon \calY_2 \to \calX_2$ two lumping maps, the lumpable families $\calW_{\kappa_1}(\calY_1, \calE_1)$ and $\calW_{\kappa_2}(\calY_2, \calE_2)$ are equivalent, which we denote
\begin{equation*}
    \calW_{\kappa_1}(\calY_1, \calE_1) \cong \calW_{\kappa_2}(\calY_2, \calE_2),
\end{equation*}
whenever there exist two bijections $\phi^\sharp \colon \calY_1 \to \calY_2$ and $\phi^\flat \colon \calX_1 \to \calX_2$ such that
\begin{equation*}
\begin{split}
\forall (y,y') \in \calY_1^2, 
    (\phi^\sharp(y),\phi^\sharp(y')) \in \calE_2 &\Longleftrightarrow (y,y') \in \calE_1, \\
    \forall (x,x') \in \calX_1^2, 
    (\phi^\flat(x),\phi^\flat(x')) \in \calD_2 = \kappa(\calE_2), &\Longleftrightarrow (x,x') \in \calD_1 = \kappa(\calE_1), \\
   \forall y \in \calY_1, \phi^\flat(\kappa_1(y)) &=  \kappa_2(\phi^\sharp (y)).
\end{split}
\end{equation*}

\subsection{Exponential families of stochastic matrices}

\begin{definition}[$\stoch$-normalization]
\label{definition:s-normalization}
When $(\calY, \calE)$ is strongly connected we define the mapping
\begin{equation*}
\begin{split}
    \label{nom:stochastic-rescaling}\stoch \colon \calF^+(\calY,\calE) &\to \calW(\calY,\calE) \\
    F &\mapsto P \colon \calE \to \bbR_+, (y,y') \mapsto P(y,y') = \frac{F(y,y') v_F(y')}{\rho_F v_F(y)},
\end{split}
\end{equation*}    
where $\rho_F$ and $v_F$ are respectively the Perron--Frobenius (PF) root and associated right eigenvector of $F$. Henceforth, $(\rho_F, v_F)$ will be called the right PF eigen-pair of $F$.
\end{definition}

The above-defined operation normalizes an arbitrary non-negative irreducible matrix into a  stochastic matrix \citep{miller1961convexity}.
 Essentially, $\stoch$-normalization plays the role of dividing by the potential function in the distribution setting.

\begin{definition}[{\citep[Section~3]{nagaoka2005exponential}}]
\label{definition:anti-shift-functions}
Let
\begin{equation*}
\begin{split}\label{nom:antishift-plus-constant}
\calN(\calY, \calE) \eqdef \bigg\{ &N \in \calF(\calY, \calE) \colon \exists (c, f) \in (\bbR, \bbR^\calY), \forall (y, y') \in \calE, N(y, y') = f(y') - f(y) + c \bigg\},
\end{split}
\end{equation*}
and observe that 
$\mathcal{N}(\calY, \calE)$ forms a $\abs{\calY}$-dimensional vector subspace.
\end{definition}

\begin{definition}[e-family of stochastic matrices {\citep{nagaoka2005exponential}}]
\label{definition:e-family}
    We say that the parametric family of irreducible stochastic matrices $$\calV_e = \set{P_\theta \colon \theta = (\theta^1, \dots, \theta^d) \in \bbR^d} \subset \calW( \calY, \calE),$$ 
forms an exponential family (e-family) of stochastic matrices with natural parameter $\theta$ and dimension $d$, when
there exist a function $K \in \calF(\calY, \calE)$ and $d$ linearly independent functions\footnote{Although $G_1, \dots, G_d$ are technically cosets, we treat them as coset representatives in $\calF(\calY, \calE)$.} $G_1, \dots, G_d \in \calG(\calY, \calE)$, 
such that
    \begin{equation*}
        P_\theta = \stoch \circ \exp \left(K + \sum_{i = 1}^{d} \theta^i G_i\right),
    \end{equation*}
    where $\mathcal{G}(\calY, \calE)$ is the quotient space
\begin{equation*}\label{nom:quotient-space}
    \calG(\calY, \calE) \eqdef \calF(\calY, \calE)/\calN(\calY, \calE),
\end{equation*}
with $\calN(\calY, \calE)$ introduced in Definition~\ref{definition:anti-shift-functions}
    and $\stoch$-normalization follows from  Definition~\ref{definition:s-normalization}.
\end{definition}

In other words, there exists a one-to-one correspondence between linear subspaces of $\calG(\calY, \calE)$ and e-families \citep[Theorem~2]{nagaoka2005exponential} through the diffeomorphism
\begin{equation*}
\begin{split}
    \stoch \circ \exp \colon \calG(\calY, \calE) &\to \calW(\calY, \calE).
    \end{split}
\end{equation*}

\begin{remark}
    Note that in Definition~\ref{definition:e-family}, the parameter space is always taken to be $\bbR^d$; however, one could in principle define exponential families on a general convex subset of $\bbR^d$.
\end{remark}

Similarly, a mixture family (m-family) of stochastic matrices is induced from the affine hull of a collection of irreducible edge measures \citep{nagaoka2005exponential} (refer also to \citet[Section~4.2]{ hayashi2014information}).
An m-family which also forms an e-family is called an em-family.
For instance, the set of all irreducible stochastic matrices $\calW(\calY, \calE)$ 
is known to form an em-family \citep{nagaoka2005exponential}.
\begin{example}[Birth-and-death-chains]
 Let $\calY = \set{1, 2, \dots, k}$ be the set of integers up to $k$ and 
$$\calE = \set{ (y, y') \in \calY^2 \colon |y - y'| \leq 1 }.$$ 
A Markov chain with a transition matrix  $P \in \calW(\calY, \calE)$ is referred to as a birth-and-death-chain, and $\calW(\calY, \calE)$ forms both an m-family and an e-family. In fact, $\calW(\calY, \calE)$ is an example of a reversible e-family \citep[Example~2]{wolfer2021information}.
\end{example}
What is more, the reversible subset is also an em-family \citep{wolfer2021information} while the subset of bistochastic matrices forms an m-family but does not form an e-family \citep{hayashi2014information}.
For a more comprehensive introduction to the information geometry of Markov chains, we refer the reader to the recent survey by \citet{wolfer2023information}.

\subsection{Foliation on the \texorpdfstring{$\kappa$-lumpable}{lumpable} family.}

Although the lumpable family $\calW_{\kappa}(\calY, \calE)$ was shown by \citet{wolfer2024geometric} to generally not form an m-family or an e-family of stochastic matrices, it is always possible to obtain a decomposition it in terms of simpler mathematical structures, called a foliation.\footnote{A foliation is a decomposition of a manifold into a union of connected but disjoint submanifolds, called leaves, all sharing the same dimension \citep[Chapter~19]{lee2013smooth}.}
This decomposition is facilitated by the notion of a Markov embedding \citep[Definition~4.3]{wolfer2024geometric}, which corresponds to a right inverse of the lumping operation and satisfies additional natural structural constraints. As such, Markov embeddings are the counterparts of the statistical morphisms axiomatized by \citet{cencov1983statistical} in the context of stochastic matrices.
In particular, any $P \in \calW_{\kappa}(\calY, \calE)$ uniquely induces a canonical embedding \citep[Lemma~4.8]{wolfer2024geometric} denoted $\Lambda_\star^{(P)} \colon \calW(\calX, \calD) \to \calW_{\kappa}(\calY, \calE)$ satisfying $P = \Lambda_\star^{(P)} \kappa_\star P$, where $\kappa_\star$ is the push-forward of the lumping map $\kappa$ (refer to \eqref{eq:push-forward-lumping}).
It was established that for any $P^\sharp_0 \in \calW(\calY, \calE)$, the embedding of the family $\calW(\calX, \calD)$ by $\Lambda_\star^{(P^\sharp_0)}$
\begin{equation*}
    \calJ(P^\sharp_{0}) \eqdef \set{ \Lambda_\star^{(P^\sharp_{0})}P \colon P \in \calW(\calX, \calD) },
\end{equation*}
forms an e-family of stochastic matrices.
Additionally, for any $P_0^\flat \in \calW(\calX, \calD)$, the set of stochastic matrices lumping into $P^{\flat}_{0}$,
\begin{equation*}
    \calL(P^{\flat}_{0}) \eqdef \set{ P \in \calW_\kappa(\calY, \calE) \colon \kappa_\star P = P^{\flat}_{0}},
\end{equation*}
form an m-family, and
the manifold $\calW_{\kappa}(\calY, \calE)$ can be endowed with the following e-foliation structure.

\begin{proposition}[Foliation on $\calW_{\kappa}(\calY, \calE)$  {\citep[Theorem~6.4]{wolfer2024geometric}}]
\label{proposition:foliation-of-lumpable-kernels}
For any fixed $P^\flat_0 \in \calW(\calX, \calD)$,
\begin{equation*}
    \calW_\kappa(\calY, \calE) = \biguplus_{P \in \calL(P^\flat_0)} \calJ(P),
\end{equation*}
\begin{equation*}
    \dim \calW_\kappa(\calY, \calE) = \abs{\calE} - \sum_{(x,x') \in \calD} \abs{\calS_x} + \abs{\calD} - \abs{\calX}.
\end{equation*}
\end{proposition}

Mutually dual foliations and mixed coordinate systems play a significant role in information geometry \citep[Section~3.7]{amari2007methods}.
In particular, similar to the probability distribution setting,
the foliated structure of $\calW_\kappa(\calY, \calE)$ can be interpreted
in the context of maximum likelihood estimation of embedded models of Markov chains.
We refer the reader to \cite[Section~6.2]{wolfer2024geometric} for a discussion and additional applications.

\paragraph{Problem statement---}
Our goal is to obtain a full characterization of exponential families of lumpable stochastic matrices in terms of 
properties of the connection graph $(\calY, \calE)$ and the lumping map $\kappa$. Ideally, we wish to develop necessary and sufficient conditions which are all verifiable in polynomial time.

\paragraph{Motivation---}

Most previous studies of the geometric structure of well-known families  of stochastic matrices (for instance reversible, bistochastic or memoryless) have established geometric structures that are valid for general edge sets.
This is in stark contrast to lumpable stochastic matrices, where---perhaps surprisingly---the nature of the family depends intricately on the structure of the edge set and its interplay with the lumping map.
Another exception is found in
\citep{takeuchi2007exponential, takeuchi2017information},
where the e-family nature of the context tree
depends on some additional structural properties of the tree, which partially motivated our question.

In addition, exponential families of stochastic matrices enjoy distinctive properties that may offer analytical power to the practitioner.
For instance, the asymptotic variance of a function $G \in \bbR^\calE$ with respect to some irreducible stochastic matrix $K$ can be expressed using the second derivative of the potential function of the one-parameter exponential family (Definition~\ref{definition:e-family}) anchored at $K$ and tilted by $G$ \citep{hayashi2014information}.
Furthermore, when $\calW_{\kappa}(\calY, \calE)$ forms an e-family, Bregman geometry yields a Pythagorean identity \citep{hayashi2014information}.
Specifically, for any $P \in \calW(\calY, \calE)$ and  $\overline{P} \in \calW_\kappa(\calY, \calE)$,
\begin{equation*}
    \kl{P}{\overline{P}} = \kl{P}{P_m} + \kl{P_m}{\overline{P}}
\end{equation*}
where 
\begin{equation*}
    P_m \eqdef \argmin_{\widetilde{P} \in \calW_{\kappa}(\calY, \calE)} \kl{P}{\widetilde{P}}
\end{equation*}
is the unique m-projection (reverse information projection) of $P$ onto $\calW_{\kappa}(\calY, \calE)$. Here $D$ is the Kullback--Leibler divergence rate defined for $P, P' \in \calW(\calY, \calE)$ as
\begin{equation*}
    \kl{P}{P'} \eqdef \sum_{(y,y') \in \calE} \pi(y) P(y,y') \log \frac{P(y,y')}{P'(y,y')},
\end{equation*}
where $\pi$ is the unique stationary distribution of $P$.

\paragraph{First approach---} As Markov embeddings are known to be e-geodesic affine \citep[Theorem~10]{wolfer2024geometric}, an immediate sufficient condition for a family $\calW_\kappa(\calY, \calE)$ to be an e-family consists in proving the existence of an embedding $\Lambda_\star$ satisfying $\calW_\kappa(\calY, \calE) = \Lambda_\star \calW(\calX, \calD)$. This corresponds to restricting the foliation of Proposition~\ref{proposition:foliation-of-lumpable-kernels} to a single e-leaf $\calJ$. As we will see in this paper, this condition is quite restrictive; there exist many more exponential families.

\section{The lumpable cone}
\label{section:lumpable-cone}
Although the purpose of this paper is to provide methods for determining whether $\calW_\kappa(\calY, \calE)$ forms an e-family, we first consider a related problem---the original problem without $\stoch$-normalization---due to its simpler structure.
Let $(\calY, \calE)$ be strongly connected and $\kappa \colon \calY \to \calX$ be a surjective lumping function.
Similar to \citet{wolfer2024geometric}, we define the set of lumpable functions as follows,
\begin{equation*}\label{nom:lumpable-functions}
    \calF_\kappa(\calY, \calE) \eqdef \set{ F \in \calF(\calY, \calE) \colon \forall (x,x') \in \calD, \forall y_1, y_2 \in \calS_{x}, F(y_1,\calS_{x'}) =  F(y_2, \calS_{x'}) }.
\end{equation*}
The positive subset \label{nom:positive-lumpable-functions}$\calF^{+}(\calY, \calE)$ of $\calF(\calY, \calE)$, forms a blunt\footnote{A convex cone is called blunt if it does not contain the null vector.} convex cone, while $\calF^{+}_\kappa(\calY, \calE)$ is a subcone of $\calF^{+}(\calY, \calE)$, as depicted on Figure~\ref{figure:cone-section}.

\begin{proposition}[Commutativity]
\label{proposition:stochastic-rescaling-and-lumping-commute-and-partition-constant-eigenvector}
The operation of $\stoch$-normalization satisfies the following properties.
\begin{enumerate}
    \item[\enumi] $\stoch$-normalization preserves $\kappa$-lumpability and 
the following diagram commutes
\[ \begin{tikzcd}
  \calF_\kappa^+(\calY, \calE) \arrow{r}{\kappa_\star} \arrow[swap]{d}{\stoch} & \calF^+(\calX, \calD)  \arrow{d}{\stoch} \\%
\calW_\kappa(\calY, \calE) \arrow{r}{\kappa_\star}&  \calW(\calX, \calD).
\end{tikzcd}
\]
\item[\enumii] For any $F \in \calF^+_\kappa(\calY, \calE)$, the right Perron--Frobenius eigenvector $v$ of $F$ satisfies that for any $x \in \calX$ and any $y_1,y_2 \in \calS_x$, $v(y_1) = v(y_2)$---we say that $v$ is constant on the partition induced from $\kappa$.
\end{enumerate}
\end{proposition}

\begin{proof}

Let $F \in \calF_\kappa^{+}(\calY, \calE)$. Since $(\calY, \calE)$ is strongly connected, $\stoch$-normalization is well-defined for $F \in \calF_\kappa^{+}(\calY, \calE)$ and we first show that $\stoch(F) \in \calW_\kappa(\calY, \calE)$.
By construction, $\stoch(F)$ is irreducible and shares the same support as $F$. It remains to verify lumpability.
We let $F^\flat = \kappa_\star F$ denote the lumping of $F$, and $(\rho^\flat, v^\flat)$ be the right PF eigen-pair of $F^\flat$.
A classical argument of \citet{barr1977eigenvector} (see also \citet[Lemma~12.9]{levin2009markov}) implies that $(\rho^\flat, v^\flat \circ \kappa)$ is the right PF eigen-pair of $F$. Indeed, for any $y \in \calY$, we can verify that
\begin{equation*}
    \begin{split}
        \sum_{y' \in \calY} F(y,y') v^\flat(\kappa(y')) &= \sum_{x' \in \calX} \sum_{y' \in \calS_{x'}} F(y,y') v^{\flat}(\kappa(y')) = \sum_{x' \in \calX} \left(\sum_{y' \in \calS_{x'}} F(y,y') \right) v^{\flat}(x') \\
        &= \sum_{x' \in \calX}  F^\flat(\kappa(y),x') v^{\flat}(x') = \rho^\flat v^{\flat}(\kappa(y)), \\
    \end{split}
\end{equation*}
and the claim holds by unicity of the PF root and by unicity---up to rescaling---of the associated PF right eigenvector.
As a consequence, for any $(y,y') \in \calY^2$,
\begin{equation*}
    \stoch(F)(y,y') = F(y,y') \frac{v^\flat(\kappa(y'))}{\rho^\flat v^\flat(\kappa(y))}.
\end{equation*}
For any $x \in \calX$ and any $y \in \calS_x$, it holds that
\begin{equation}
\begin{split}
\label{eq:lumpability-of-rescaling}
    \sum_{y' \in \calS_{x'}}\stoch(F)(y,y') &= \sum_{y' \in \calS_{x'}} F(y,y') \frac{v^\flat(\kappa(y'))}{\rho^\flat v^\flat(\kappa(y))}
    = \sum_{y' \in \calS_{x'}} F(y,y') \frac{v^\flat(x')}{\rho^\flat v^\flat(x)} = F^\flat(x,x') \frac{v^\flat(x')}{\rho^\flat v^\flat(x)},
\end{split}
\end{equation}
which does not depend on $y \in \calS_x$. As a result, $\stoch(F)$ is $\kappa$-lumpable, and \eqref{eq:lumpability-of-rescaling} also means that $\kappa_\star \stoch(F) = \stoch(F^\flat)$.  
\end{proof}

\begin{figure}
\begin{center}

\tikzset{every picture/.style={line width=0.75pt}} %

\begin{tikzpicture}[x=0.52pt,y=0.52pt,yscale=-1,xscale=1]

\draw  [color={rgb, 255:red, 208; green, 2; blue, 27 }  ,draw opacity=1 ] (198,77) .. controls (198,59.33) and (248.37,45) .. (310.5,45) .. controls (372.63,45) and (423,59.33) .. (423,77) .. controls (423,94.67) and (372.63,109) .. (310.5,109) .. controls (248.37,109) and (198,94.67) .. (198,77) -- cycle ;
\draw [color={rgb, 255:red, 208; green, 2; blue, 27 }  ,draw opacity=1 ]   (198,79) -- (218.2,113.79) -- (310,281) ;
\draw [color={rgb, 255:red, 208; green, 2; blue, 27 }  ,draw opacity=1 ]   (310,281) -- (422,81) ;
\draw  [color={rgb, 255:red, 74; green, 144; blue, 226 }  ,draw opacity=1 ] (269,204) .. controls (269,194.06) and (287.58,186) .. (310.5,186) .. controls (333.42,186) and (352,194.06) .. (352,204) .. controls (352,213.94) and (333.42,222) .. (310.5,222) .. controls (287.58,222) and (269,213.94) .. (269,204) -- cycle ;
\draw [color={rgb, 255:red, 208; green, 2; blue, 27 }  ,draw opacity=1 ] [dash pattern={on 0.84pt off 2.51pt}]  (160,9) -- (208.1,96.4) ;
\draw [color={rgb, 255:red, 208; green, 2; blue, 27 }  ,draw opacity=1 ] [dash pattern={on 0.84pt off 2.51pt}]  (422,81) -- (461,11) ;
\draw  [color={rgb, 255:red, 74; green, 144; blue, 226 }  ,draw opacity=1 ][dash pattern={on 4.5pt off 4.5pt}] (212.3,162.5) -- (556,162.5) -- (408.7,245.5) -- (65,245.5) -- cycle ;
\draw  [dash pattern={on 0.84pt off 2.51pt}]  (250,9) -- (310,281) ;
\draw  [fill={rgb, 255:red, 0; green, 0; blue, 0 }  ,fill opacity=1 ] (267,71) .. controls (267,69.34) and (265.66,68) .. (264,68) .. controls (262.34,68) and (261,69.34) .. (261,71) .. controls (261,72.66) and (262.34,74) .. (264,74) .. controls (265.66,74) and (267,72.66) .. (267,71) -- cycle ;
\draw  [fill={rgb, 255:red, 0; green, 0; blue, 0 }  ,fill opacity=1 ] (296,204) .. controls (296,202.34) and (294.66,201) .. (293,201) .. controls (291.34,201) and (290,202.34) .. (290,204) .. controls (290,205.66) and (291.34,207) .. (293,207) .. controls (294.66,207) and (296,205.66) .. (296,204) -- cycle ;
\draw  [fill={rgb, 255:red, 0; green, 0; blue, 0 }  ,fill opacity=1 ] (285,153) .. controls (285,151.34) and (283.66,150) .. (282,150) .. controls (280.34,150) and (279,151.34) .. (279,153) .. controls (279,154.66) and (280.34,156) .. (282,156) .. controls (283.66,156) and (285,154.66) .. (285,153) -- cycle ;

\draw (418,110) node [anchor=north west][inner sep=0.75pt]  [color={rgb, 255:red, 208; green, 2; blue, 27 }  ,opacity=1 ] [align=left] {$\calF^+_\kappa(\calY, \calE)$};
\draw (540,170) node [anchor=north west][inner sep=0.75pt]  [color={rgb, 255:red, 74; green, 144; blue, 226 }  ,opacity=1 ] [align=left] {$\calW(\calY, \calE)$};
\draw (352,202) node [anchor=north west][inner sep=0.75pt]  [color={rgb, 255:red, 74; green, 144; blue, 226 }  ,opacity=1 ] [align=left] {$\calW_\kappa(\calY, \calE)$};
\draw (260,142) node [anchor=north west][inner sep=0.75pt]   [align=left] {$F$};
\draw (300,192) node [anchor=north west][inner sep=0.75pt]   [align=left] {$\stoch(F)$};
\draw (216,3) node [anchor=north west][inner sep=0.75pt]   [align=left] {$[F]$};
\draw (266,63) node [anchor=north west][inner sep=0.75pt]   [align=left] {$G: \stoch(G) = \stoch(F)$};

\end{tikzpicture}
\end{center}
\caption{We can regard $\calW_\kappa(\calY, \calE)$ as a section of the cone $\calF^+_\kappa(\calY, \calE)$.}
\label{figure:cone-section}
\end{figure}

For $F \in \calF^+(\calY, \calE)$ the equivalence class (depicted on Figure~\ref{figure:cone-section})
\begin{equation*}
    [F] \eqdef \set{ G \in \calF^{+}(\calY, \calE) \colon \stoch(G) = \stoch(F) },
\end{equation*}
can be parametrized by the $\abs{\calY}$-dimensional positive orthant,\footnote{Without loss of generality, we can fix the $\ell_1$ norm of the right eigenvector. As a result, we obtain $\abs{\calY} - 1$ positive degrees of freedom from $v$ and an extra positive degree of freedom from the scalar $\rho$.}
\begin{equation*}
    [F] = \set{  \rho \diag(v) F \diag(v)^{-1} \colon \rho \in \bbR_+, v \in \bbR_+^{\calY}},
\end{equation*}
where $\diag(v)$ denotes the diagonal matrix with the entries of $v$ on its diagonal,
and the collection of all such rays generates the entire irreducible cone
\begin{equation*}
    \calF^+(\calY, \calE) = \biguplus_{P \in \calW(\calY, \calE)} [P].
\end{equation*}
While projecting onto stochastic matrices through $\stoch$-normalization preserves $\kappa$-lumpability, conjugation by an arbitrary diagonal matrix can disrupt the property. In fact, the following proposition holds.
\begin{proposition}[Closure under similarity transform]
\label{proposition:eigenvector-rescaling-constant-on-each-lumped-point}
    Let $v \in \bbR_+^{\calY}$. The following two statements are equivalent.
\begin{enumerate}[$(i)$]
    \item[\enumi] For any  irreducible lumpable matrix $F \in \calF_{ \kappa}^+(\calY, \calE)$, $v$ satisfies \begin{equation*}
\diag(v) F \diag(v)^{-1} \in \calF_{\kappa}^{+}(\calY, \calE).
    \end{equation*}
    \item[\enumii] 
    For any $x \in \calX$, $v$ satisfies  $v(y_1) = v(y_2)$ for any $y_1, y_2 \in \calS_{x}$.
\end{enumerate}
\end{proposition}

\begin{proof}
We first prove that \enumii implies \enumi. Let $F \in \calF_\kappa^+(\calY, \calE)$, write $F^\flat = \kappa_\star F $, let $(x,x') \in \calD = \kappa_2(\calE)$, and assume that $v$ takes constant values $v_x$ on $\calS_x$ and $v_{x'}$ on $\calS_{x'}$. Then for any $y \in \calS_x$,
    \begin{equation*}
    \begin{split}
        \sum_{y' \in \calS_{x'}} (\diag(v) F \diag(v)^{-1})(y, y') &= \frac{v_x}{v_{x'}} \sum_{y' \in \calS_{x'}} F(y, y') = \frac{v_x}{v_{x'}} F^\flat(x,x'),
    \end{split}
    \end{equation*}
    which does not depend on $y$,
    thus $\diag(v) F \diag(v)^{-1} \in \calF_\kappa^{+}(\calY, \calE)$. 
    Conversely, let us assume that \enumi holds, and let $F \in \calF_\kappa^+(\calY, \calE)$. We thus have $\diag(v) F \diag(v)^{-1} \in \calF_\kappa^{+}(\calY, \calE)$.
    Let $\rho \in \bbR, w \in \bbR_+^{\calY}$ be the unique right PF eigen-pair of $F$. 
    From  Proposition~\ref{proposition:stochastic-rescaling-and-lumping-commute-and-partition-constant-eigenvector}--\enumii, $w$ must be constant on each $\calS_x$,
    and it holds that,
    \begin{equation*}
        \begin{split}
            \stoch(F) &= \frac{1}{\rho} \diag(w)^{-1} F \diag(w) \\
            &= \frac{1}{\rho} \diag(v \odot w)^{-1} \diag(v) F \diag(v)^{-1} \diag(v \odot w) \\
            &= \stoch( \diag(v) F \diag(v)^{-1} ).
        \end{split}
    \end{equation*}
    It follows 
    from Proposition~\ref{proposition:stochastic-rescaling-and-lumping-commute-and-partition-constant-eigenvector}--\enumii
    that $v \hadamard w$ must be constant over each $\calS_x$, and so must $v$.
\end{proof}

In other words, the cone $\calF_\kappa^{+}(\calY, \calE)$ is closed with respect to similarity transform with positive vectors which are constant on the partition induced from $\kappa$.

\begin{corollary}[to Proposition~\ref{proposition:stochastic-rescaling-and-lumping-commute-and-partition-constant-eigenvector}]
\label{corollary:geometric-intersection-interpretation}
    Let $F \in \calF^+(\calY, \calE)$. Then $\stoch(F) \in \calW_\kappa(\calY, \calE)$ if and only if $$[F] \cap \calF_\kappa^{+}(\calY, \calE) \neq \emptyset.$$
\end{corollary}

\begin{proof}
    Suppose first that $\stoch(F) \in \calW_\kappa(\calY, \calE)$. Clearly, $F \in [F]$, and since $\calW_\kappa(\calY, \calE) \subset \calF_\kappa^{+}(\calY, \calE)$ it is immediate that $\stoch(F) \in [F] \cap \calF_\kappa^{+}(\calY, \calE) \neq \emptyset$.
    Conversely, suppose that there exists $G \in [F]$ such that $G \in \calF_\kappa^+(\calY, \calE)$. Then, by Proposition~\ref{proposition:stochastic-rescaling-and-lumping-commute-and-partition-constant-eigenvector}--\enumi, $\stoch(G) \in \calW_\kappa(\calY, \calE)$, but $\stoch(G) = \stoch(F)$ since both are in $[F]$, thus also $\stoch(F) \in \calW_\kappa(\calY, \calE)$.
\end{proof}

The notion of a merging block, which we now introduce, will be instrumental to our analysis.

\begin{definition}[Merging block]
\label{definition:merging-block}
Let $(x, x') \in \calD = \kappa_2(\calE)$ be such that for some $y \in \calS_x$ it holds that
\begin{equation}
\label{eq:merging-block-condition}
    \abs{\set{(y,y') \in \calE \colon y' \in \calS_{x'}}} > 1.
\end{equation}
Then we say that $(x,x')$ is a merging block of $(\calY, \calE)$ with respect to $\kappa$.
Furthermore, when 
$\abs{\calS_x} \geq 2$, we say that the merging block is multi-row (note that \eqref{eq:merging-block-condition} need not hold for every $y \in \calS_x$).
\end{definition}

It will be convenient to define
\begin{equation*}
    \calG_\kappa(\calY, \calE) \eqdef \set{ \log F \colon F \in \calF^+_\kappa(\calY, \calE) },
\end{equation*}
where we recall that the logarithm is understood to be entry-wise.

\begin{theorem}[Log-affinity of $\calF_\kappa^{+}(\calY, \calE)$]
\label{theorem:characterization-log-linearity}
The following two statements are equivalent.

\begin{enumerate}[$(i)$]
    \item[\enumi]
$(\calY, \calE)$ has no multi-row merging block with respect to $\kappa$ (Definition~\ref{definition:merging-block}). 
\item[\enumii] $\calG_\kappa(\calY, \calE)$ is affine.
\end{enumerate}
\end{theorem}

\begin{proof}
Let us first show that \enumi implies \enumii.
    Let $F_0, F_1 \in \calF_\kappa^{+}(\calY, \calE)$, let $t \in \bbR$, and write $F_0^\flat = \kappa_\star F_0, F_1^\flat = \kappa_\star F_1$. For $(y,y') \in \calY^2$, we define $F_t(y,y') = F_0(y,y')^{1 - t} F_1(y,y')^t$.
Let $(x,x') \in \calD$. We need to show that for any $y_1, y_2 \in \calS_{x}$, we have $F_t(y_1, \calS_{x'}) = F_t(y_2, \calS_{x'})$. The case $\abs{\calS_x} = 1$ is trivial and it remains to inspect the case $\abs{\calS_x} \geq 2$. Since $(x,x')$ is non-merging, for any $y \in \calS_{x}$ we denote $u'_{x'}(y)$ the unique element in $\calS_{x'}$ such that $(y,u'_{x'}(y)) \in \calE$. For any $y \in \calS_x$, it then holds that
\begin{equation*}
\begin{split}
    \sum_{y' \in \calS_{x'}} F_t(y,y') &= \sum_{y' \in \calS_{x'}} F_0(y,y')^{1 -t} F_1(y,y')^{t} \\ &= F_0(y,u'_{x'}(y))^{1-t} F_1(y,u'_{x'}(y))^{t} \\
    &= \left(\sum_{y'\in \calS_{x'}} F_0(y,y') \right)^{1-t} \left( \sum_{y' \in \calS_{x'}} F_1(y,y') \right)^{t}
    = F_0^\flat(x,x') ^{1-t}   F_1^\flat(x,x')^{t},
\end{split}
\end{equation*}
which does not depend on $y$, thus $F_t \in \calF_\kappa^+(\calY, \calE)$.
Conversely, to show that \enumii implies \enumi, suppose that $(\calY, \calE)$ has a multi-row merging block with respect to $\kappa$. We show the somewhat stronger claim that there exist two stochastic matrices $P_0, P_1 \in \calW_\kappa(\calY, \calE)$ and $t \in \bbR$ such that $P_0^{\hadamard (1 - t)} \hadamard P_1^{\hadamard t} \not \in \calF_{\kappa}^{+}(\calY, \calE)$.
    We let $(x_0,x_0') \in \calD$ be  a merging block of $(\calY, \calE)$, and  denote $y_0, y_\star \in \calS_{x_0}, y'_a, y'_b \in \calS_{x_0'}$ such that $y_0 \neq y_\star$, $y'_a \neq y'_b$ and $(y_\star, y'_a), (y_\star, y'_b) \in \calE$.
    Let $\eta_a, \eta_b \in \bbR_+$ with $\eta_a < \eta_b < 1$ and such that 
    $$\eta_a + \eta_b = 2\abs{ \set{\overline{x} \in \calX \colon (x_0, \overline{x}) \in \calD }}^{-1} \abs{ \set{ \overline{y} \in \calS_{x_0'} \colon (y_\star,\overline{y}) \in \calE }}^{-1}.$$
    We construct,
    \begin{equation*}
        P_{a,b}(y,y') = \begin{cases}     
            0 &\text{when } (y, y') \not \in \calE \\
            \eta_a &\text{when } (y,y') = (y_\star, y'_a) \\
            \eta_b &\text{when } (y,y') = (y_\star, y'_b) \\
            \cfrac{1}{\abs{ \set{\overline{x} \in \calX \colon (\kappa(y), \overline{x}) \in \calD }} \abs{ \set{ \overline{y} \in \calS_{\kappa(y')} \colon (y,\overline{y}) \in \calE }}} &\text{otherwise},
        \end{cases}
    \end{equation*}
    so that $\kappa_\star P_{a,b}(x,\cdot)$ is the uniform distribution on $\{\overline{x} \in \calX: (x,\overline{x}) \in \calD \}$.
    We define $P_0, P_1 \in \calW_\kappa(\calY, \calE)$ as $P_0 = P_{a,b}$ and $P_1 = P_{b,a}$,
    and construct the combination $\widetilde{P} = P_0^{\hadamard 1/2} \hadamard P_1^{\hadamard 1/2}$.
    By the AM-GM inequality $2 \sqrt{\eta_a \eta_b} < \eta_a + \eta_b$, and it follows that
    \begin{equation*}
        \sum_{\overline{y} \in \calS_{x_0'}} \widetilde{P}(y_\star, \overline{y}) < \sum_{\overline{y} \in \calS_{x_0'}} \widetilde{P}(y_0, \overline{y})
    \end{equation*}
    hence $\widetilde{P}_{1/2}$ is not $\kappa$-lumpable.
\end{proof}

In words, the log-affinity of the cone of positive lumpable functions is characterized by the absence of multi-row merging blocks, providing a solution to our simplified problem. In the next section, we will see that the answer to the original problem, with $\stoch$-normalization, is more involved.

\section{Characterizing of e-families of lumpable stochastic matrices}
\label{section:towards-characterization}
In this section, we build upon the results of Section~\ref{section:lumpable-cone} to analyze the original problem of characterizing lumpable e-families. Let us assume that $1 < \abs{\calX} < \abs{\calY}$ and that $\calW_\kappa(\calY, \calE) \neq \emptyset$.

\subsection{Sufficient combinatorial conditions}
\label{section:sufficient-condition}

\subsubsection{No multi-row merging block criterion}

The first criterion is a natural consequence of Theorem~\ref{theorem:characterization-log-linearity}.

\begin{corollary}[No multi-row merging block criterion]
\label{corollary:no-multi-row-merging-block-is-sufficient}
If $(\calY, \calE)$ has no multi-row merging block with respect to $\kappa$, then $\calW_\kappa(\calY, \calE)$ forms an e-family.
\end{corollary}

\begin{proof}
    Let $P_0, P_1 \in \calW_\kappa(\calY, \calE)$ and $t \in \bbR$. Defining $\widetilde{P}_t = P_0^{\hadamard (1 - t)} \hadamard P_1^{\hadamard t}$, it follows from Theorem~\ref{theorem:characterization-log-linearity} that $\widetilde{P}_t \in \calF_{\kappa}^{+}(\calY, \calE)$.
    It is then a consequence of Proposition~\ref{proposition:stochastic-rescaling-and-lumping-commute-and-partition-constant-eigenvector}--\enumi that $\stoch(\widetilde{P}_t) \in \calW_\kappa(\calY, \calE)$. Since $P_0, P_1$ and $t$ were arbitrary, \citet[Corollary~3]{nagaoka2005exponential} implies that $\calW_\kappa(\calY, \calE)$ forms an e-family.
\end{proof}

\begin{example}[Hudson expansion {\citep{kemeny1983finite}}]
Let $(\calX, \calD)$ be a strongly connected graph. 
For a Markov chain $X_1, X_2, \dots$ sampled according to a transition matrix $P \in \calW(\calX, \calD)$, recall that the sliding window chain
$$(X_1, X_2), (X_2, X_3), \dots, (X_t, X_{t+1}), \dots$$ is also a Markov chain with transition matrix $P^\sharp \in \calW_h(\calY, \calE)$, with state space $\calY = \calD$, edge set
\begin{equation*}
    \calE = \set{ (e = (x_1, x_2),e' = (x_1', x_2')) \in \calD^2 \colon x_2 = x_1' },
\end{equation*} 
and lumping function $h \colon \calY \to \calX, (x_1, x_2) \mapsto x_2$.
One can verify that $(\calY, \calE)$ has no multi-row merging block with respect to $h$.
As a result of Corollary~\ref{corollary:no-multi-row-merging-block-is-sufficient}, $\calW_{h}(\calY, \calE)$ thus forms an e-family of lumpable Markov chains.
We therefore recover the known fact that the Hudson expansion of the first-order Markov chains forms an e-subfamily of second-order Markov chains.
Note that, since the Hudson expansion is known to be a particular case of a Markov embedding \citep{wolfer2024geometric}, and Markov embeddings are known to preserve e-families of stochastic matrices, the claim also follows from an embedding argument.
\end{example}

\subsubsection{Lazy-cycle criterion}

 Whilst the condition in Corollary~\ref{corollary:no-multi-row-merging-block-is-sufficient} is also necessary for state spaces of size at most three (refer to Section~\ref{section:classification}), it is no longer necessary for larger state spaces, as demonstrated in this section. This is a consequence of the fact that constructing an e-geodesic further involves $\stoch$-normalization which can return curves to the lumpable set. The following proposition states that when all diagonal blocks are diagonal and there is exactly one non-zero off-diagonal block per block-line, the lumpable family is exponential---regardless of the off-diagonal blocks being multi-row merging.

\begin{proposition}[Lazy cycle criterion]
    \label{proposition:lazy-cycle-criterion}
    If any of the following two equivalent conditions are satisfied, then $\calW_{\kappa}(\calY, \calE)$ forms an e-family.
    \begin{enumerate}
        \item[\enumi] For any $P \in \calW_{\kappa}(\calY, \calE)$, there exist a pair of non-negative matrices $D$ and $\Pi$ such that
    \begin{equation*}
        P = D + \Pi,
    \end{equation*}
    where $\kappa_\star \Pi$ is a permutation matrix over $\calX$ and $D$ is diagonal.
    \item[\enumii] For any $(y,y') \in \calE, \kappa(y) = \kappa(y')$ implies that $ y = y'$ and the graph $$\left( \calX,   \set{ (\kappa(y),\kappa(y')) \colon (y,y') \in \calE, \kappa(y) \neq \kappa(y')}  \right)$$ is a cycle.
    \end{enumerate}
\end{proposition}

\begin{proof}
We first prove the statement in the special case where all diagonal blocks vanish, that is $\{ (x,x) \colon x \in \calX \} \cap \calD = \emptyset$, and when the off-diagonal blocks have full support, that is for any $(x,x') \in \calD$ with $x \neq x'$, $\calS_{x} \times \calS_{x'} \subset \calE$. An application of monotonicity (refer to Theorem~\ref{theorem:monotonicity}) generalizes the result beyond full-sport off-diagonal blocks, and stability by diagonal modifications (refer to Proposition~\ref{proposition:lazy-cycle-criterion}) yield the more general case where there are non vanishing diagonal blocks on the diagonal.
We further reduce the problem by observing that since the lumped matrix is irreducible, $\Pi^\flat = \kappa_\star \Pi$ defines a cycle.
Finally, it will be convenient to order states $\calX = \set{1, \dots, \abs{\calX}}$ and $\calY = \set{1, \dots, \abs{\calY}}$.
As a result, upon relabeling, we henceforth assume that the family can be represented by
\begin{equation*}
\left(\begin{tblr}{c|[dashed]c|[dashed]c|[dashed]c|[dashed]c}
  0 & \boxplus_{\calS_{1} \times \calS_{2}} & 0 & \hdots & 0\\\hline[dashed]
  0 & 0 & \boxplus_{\calS_{2} \times \calS_{3}} & \ddots & \vdots \\\hline[dashed]
  \vdots & & \ddots & \ddots & 0 \\\hline[dashed]
  0 & & & 0 & \boxplus_{\calS_{\abs{\calX} - 1} \times \calS_{\abs{\calX}}} \\\hline[dashed]
  \boxplus_{\calS_{\abs{\calX}} \times \calS_{1}} & 0 & \hdots & 0 & 0 
\end{tblr}\right)
\end{equation*}
where for $\calS, \calS' \subset \calY$, $\boxplus_{\calS \times \calS'}$ is $\abs{\calS} \times \abs{\calS'}$ matrix defined by
\begin{equation*}
   \textstyle{\boxplus_{\calS \times \calS'}} = \begin{pmatrix}
        \plus & \plus & \hdots & \plus \\
        \plus & \plus & \hdots & \plus \\
        \vdots & \vdots & \vdots & \vdots \\
        \plus & \plus & \hdots & \plus \\
    \end{pmatrix}.
\end{equation*}
We let $P_0, P_1 \in \calW_{\kappa}(\calY, \calE)$, and for $t\in \bbR$, we denote $\widetilde{P}_t = P_0^{\hadamard (1- t)} \hadamard P_1^{\hadamard t} \in \calF^{+}(\calY, \calE)$ their log-affine combination. We define $(\rho_t, v_t)$ the right PF pair of $\widetilde{P}_t$, whose existence follows by strong connectivity of $(\calY, \calE)$. By $\stoch$-normalization $P_t = \frac{1}{\rho_t}\diag(v)^{-1} \widetilde{P}_t \diag(v)$ is row-stochastic, 
and as a result, for any $(x,x') \in \calD$, it holds that for any $y \in \calS_{x}$,
\begin{equation*}
    \sum_{y' \in \calS_{x'}} P_t(y,y') = \sum_{y' \in \calY} P_t(y,y') = 1,
\end{equation*}
thus $P_t$ is $\kappa$-lumpable and the claim holds.
\end{proof}

\begin{example}[Lumpable e-family with two multi-row merging blocks]
\label{example:e-family-two-merging-block}

Let $\calY = \set{0,1,2,3}, \calX= \set{a,b}$, the lumping map $\kappa$ defined by the partition $\calS_{a} = \set{0,1}, \calS_{b} = \set{2,3}$, and 
consider the lumpable family
\begin{equation*}
\calW_\kappa(\calY,\calE) \sim 
    \left(\begin{tblr}{cc|[dashed]cc}
  \zero & \zero &  \zero & \plus  \\
  \zero & \zero & \plus & \plus\\\hline[dashed]
  \zero & \plus & \plus & \zero \\
  \plus & \plus & \zero & \plus \\
\end{tblr}\right).
\end{equation*}
Observe that the blocks $(a,b)$ and $(b,a)$ are multi-row merging, thus we cannot apply Corollary~\ref{corollary:no-multi-row-merging-block-is-sufficient}. However, the lumped family---ignoring the diagonal---forms a cycle. As a result,  Proposition~\ref{proposition:lazy-cycle-criterion} is applicable, and 
$\calW_\kappa(\calY,\calE)$ forms in fact an e-family.
\end{example}

\begin{corollary}
    There exist e-families of lumpable stochastic matrices with an arbitrary number of multi-row merging blocks.
\end{corollary}

It is instructive to observe that neither criterion---Corollary~\ref{corollary:no-multi-row-merging-block-is-sufficient} nor Proposition~\ref{proposition:lazy-cycle-criterion}---implies the other.
What is more, there exist e-families that remain unexplained by the two above-mentioned criteria.

Deriving a set of necessary conditions for $\calW_{\kappa}(\calY, \calE)$ to form an e-family is the topic of our next section.

\subsection{Necessary combinatorial conditions}
\label{section:necessary-conditions}

As we have seen in the previous section, absence of log-affinity does not always preclude the lumpable family from being exponential. However, the former will be a critical ingredient to show that at least some families of matrices with so called ``redundant blocks'' cannot form e-families.

\subsubsection{Redundant merging block criterion}
\label{section:redundant-merging-block}

First, it will be convenient to consider the following operations on matrices.

\begin{definition}[Sub-matrix and block removal]
\label{definition:sub-matrix-block-removal}
Let $F \in \calF_\kappa^{+}(\calY, \calE)$ and a non-degenerate surjective lumping function $\kappa \colon \calY \to \calX$.
\begin{description}
    \item[Sub-matrix.] 
    For $\calT \subset \calX$, we write
\begin{equation*}
\begin{split}
    \calY|_{\calT} \eqdef \bigcup_{x \in \calT} \calS_{x} , \qquad 
    \calE|_{\calT} \eqdef \calE \cap \left( \bigcup_{x,x' \in \calT} \calS_{x} \times \calS_{x'}\right),
\end{split}
\end{equation*}
the resulting sub-graph is  ($\calY|_{\calT}, \calE|_{\calT}$) and the sub-matrix $F|_{\calT} \in \calF_\kappa^{+}(\calY|_{\calT}, \calE|_{\calT})$ is defined such that for any $(y,y') \in \calE|_{\calT}$, \begin{equation*}
    F|_{\calT}(y,y') = F(y,y').
\end{equation*}
    \item[Block removal.]
    The matrix $F$ from which the block $(x_0, x_0') \in \kappa_2(\calE)$ has been removed is defined as 
$F^{\setminus (x_0, x_0')}\in \calF_\kappa^{+}(\calY, \calE \setminus (\calS_{x_0} \times \calS_{x'_0}))$ where for any $(y,y') \in \calE \setminus (\calS_{x_0} \times \calS_{x'_0})$,
\begin{equation*}
        F^{\setminus (x_0, x_0')}(y,y') = F(y,y').
    \end{equation*}
\end{description}
\end{definition}

Critically, while both operations in Definition~\ref{definition:sub-matrix-block-removal} preserve lumpability, any of the two can disrupt strong connectivity. When a block can be removed while in some sense preserving irreducibility, we say that it is redundant.

\begin{definition}[Redundant block]
\label{definition:redundant-block}
Let $(x_0, x_0') \in \kappa_2(\calE)$. If there exists $\calT \subset \calX$ such that the following condition holds:
    \begin{enumerate}[$(i)$]
        \item[\enumi] The block $(x_0, x_0')$ is included in $ \calT^2$.
        \item[\enumii] The sub-graph
$(\calY|_{\calT}, \calE|_{\calT} \setminus (\calS_{x_0} \times \calS_{x_0'}))$ 
      is strongly connected.
\end{enumerate}
Then we say that the block $(x_0, x_0')$ is redundant.
\end{definition}

\begin{theorem}[Redundant merging block criterion]
\label{theorem:redundant-merging-block-criterion}
    If $(\calY, \calE)$ has a multi-row merging block with respect to $\kappa$ (Definition~\ref{definition:merging-block}) which is redundant (Definition~\ref{definition:redundant-block}),
    then $\calW_{\kappa}(\calY, \calE)$ does not form an e-family.
\end{theorem}

\begin{proof}
    We suppose that $(\calY, \calE)$ has a multi-row merging block $(x_0, x_0') \in \kappa_2(\calE)$, which is redundant.
    Recall from the proof Theorem~\ref{theorem:characterization-log-linearity} that there exist stochastic matrices $P_0, P_1 \in \calW_{\kappa}(\calY, \calE)$ such that $\widetilde{P}_{1/2} = P_0^{\hadamard 1/2} \hadamard P_1^{\hadamard 1/2} \not \in \calF_{\kappa}^{+}(\calY, \calE)$.
    From Corollary~\ref{corollary:geometric-intersection-interpretation}, it is enough to show that   $[\widetilde{P}_{1/2}] \cap \calF_\kappa^{+}(\calY, \calE) = \emptyset$ in order to prove that $\stoch(\widetilde{P}_{1/2}) \not \in \calW_\kappa(\calE, \calY)$. In other words, it is sufficient to show that for any $v \in \bbR_+^{\calY}$, the rescaled matrix $\diag(v) \widetilde{P}_{1/2}\diag(v)^{-1}$ is not $\kappa$-lumpable.
    Let us suppose for contradiction that there exists $v \in \bbR_+^{\calY}$, such that $\diag(v) \widetilde{P}_{1/2}\diag(v)^{-1} \in \calF_{\kappa}^{+}(\calY, \calE)$.
    Since $(x_0,x_0')$ is redundant, there exists $\calT \subset \calX$ such that
    $(x_0, x_0') \in \calT^2$ and 
$(\calY|_{\calT}, \calE|_{\calT} \setminus (\calS_{x_0} \times \calS_{x_0'}))$ 
      is strongly connected.
    On one hand, by our assumption, it must hold that 
    \begin{equation}
    \label{eq:from-assumption}
        \diag(v) \left(\widetilde{P}_{1/2}|_{\calT}^{\setminus (x_0, x_0')}\right)\diag(v)^{-1} \in \calF_{\kappa}^{+}(\calY|_{\calT}, \calE|_{\calT} \setminus (\calS_{x_0} \times \calS_{x'_0})).
    \end{equation}
    However, observe that by our construction in the proof of Theorem~\ref{theorem:characterization-log-linearity}, 
    we obtain a matrix such that when lumped, it will be the uniform distribution on blocks other than $(x_0,x_0')$,
    and we also have
    \begin{equation}
    \label{eq:from-construction}
\widetilde{P}_{1/2}|_{\calT}^{\setminus (x_0, x_0')} \in \calF_{\kappa}^{+}(\calY|_{\calT}, \calE|_{\calT} \setminus (\calS_{x_0} \times \calS_{x'_0})).
    \end{equation} 
    It follows from irreducibility of  $(\calY|_{\calT}, \calE|_{\calT} \setminus (\calS_{x_0} \times \calS_{x'_0}))$ and an application of Proposition~\ref{proposition:eigenvector-rescaling-constant-on-each-lumped-point} that
    $v$ must be constant over each $\calS_{x}$ for $x \in \calT$ in order for both \eqref{eq:from-assumption} and \eqref{eq:from-construction} to hold, and in particular for $x \in \set{ x_0, x_0'}$.
    However, in this case, inspecting the multi-row merging block $(x_0, x_0')$,
    \begin{equation*}
        \begin{split}
\sum_{\overline{y} \in \calS_{x_0'}} (\diag(v) \widetilde{P}_{1/2} \diag(v)^{-1})(y_0, \overline{y}) &= \frac{v_{x_0}}{v_{x_0'}} \frac{1}{\abs{\overline{x} \in \calX \colon (x_0, \overline{x}) \in \calD}},
        \end{split}
    \end{equation*}
    while
    \begin{equation*}
    \begin{split}
        \sum_{\overline{y} \in \calS_{x_0'}} (\diag(v) \widetilde{P}_{1/2} \diag(v)^{-1})(y_\star, \overline{y}) &= \frac{v_{x_0}}{v_{x_0'}} \left( \frac{1}{\abs{\overline{x} \in \calX \colon (x_0, \overline{x}) \in \calD}} + 2 \sqrt{\eta_a \eta_b} - (\eta_a + \eta_b) \right),
    \end{split}
    \end{equation*}
which cannot be equal from the AM-GM inequality and the assumption that $\eta_a \neq \eta_b$.
\end{proof}

\begin{example}
Let $\calY = \set{0,1,2,3,4,5}, \calX= \set{a,b,c}$, the lumping map $\kappa$ defined by the partition $\calS_{a} = \set{0,1}, \calS_{b} = \set{2,3}, \calS_{c} = \set{4,5}$, and 
consider the lumpable family
     \begin{equation*}
     \calW_{\kappa}(\calY, \calE) \sim
\left(\begin{tblr}{cc|[dashed]cc|[dashed]cc}
  \plus & \zero & \plus & \plus & \plus & \zero \\
  \zero & \plus & \plus & \plus & \zero & \plus \\\hline[dashed]
  \plus & \zero & \plus & \zero & \plus & \plus \\
  \zero & \plus & \zero & \plus & \plus & \plus \\\hline[dashed]
  \plus & \plus & \plus & \zero & \plus & \zero \\
  \plus & \plus & \zero & \plus & \zero & \plus \\
\end{tblr}\right).
\end{equation*}
Here, $(b,c) \in \calD$ is a merging block. When we remove edges pertaining to the block $(b,c)$---that is we remove $(2, 4), (2, 5), (3, 4), (3, 5)$---there exists a closed path 
$$1 \rightarrow 2 \rightarrow 0 \rightarrow 4 \rightarrow 2 \rightarrow 0 \rightarrow 3 \rightarrow 1 \rightarrow 5 \rightarrow 3 \rightarrow 1,$$ 
going through all the states in $\calY$. As a result, the block $(b, c)$ is redundant with $\calT = \calY$ and from Theorem~\ref{theorem:redundant-merging-block-criterion},  $\calW_{\kappa}(\calY, \calE)$ does not form an e-family.
\end{example}

\begin{corollary}[Complete graph]
\label{corollary:complete-graph}
Unless $\kappa$ is degenerate, $\calW_{\kappa}(\calY, \calY^2)$ does not form an e-family.
\end{corollary}

\begin{proof}
    When $\kappa$ is non-degenerate, observe that
     $(\calY, \calY^2)$ must have a multi-row merging block. Additionally, removing this block still yields a strongly connected graph. 
    From a direct application of  Theorem~\ref{theorem:redundant-merging-block-criterion}, it follows that $\calW_{\kappa}(\calY, \calY^2)$ does not form an e-family.
\end{proof}

The most intriguing scenarios occur in the intermediate range between having no multi-row merging blocks and all blocks being multi-row merging.
Indeed, as the number of edges in the connection graph grows, it becomes increasingly difficult for the lumpable family to form an e-family. This intuition will be rigorously formalized in Section~\ref{section:monotonicity}.

\subsection{Characterization based on a dimensional criterion}
\label{section:dimensional-criterion}

Sections~\ref{section:sufficient-condition}~and~\ref{section:redundant-merging-block} focused on analyzing combinatorial properties of the connection graph. In this section we develop a more geometric argument to determine whether $\calW_\kappa(\calY, \calE)$ forms an e-family. We first construct the e-hull of $\calW_\kappa(\calY, \calE)$, that is, the smallest e-family containing $\calW_\kappa(\calY, \calE)$---this notion is akin to taking an affine hull in the e-coordinates and a formal definition will be given in Definition~\ref{definition:e-hull}. 
Intuitively but crucially, $\calW_\kappa(\calY, \calE)$ coincides with its e-hull if and only if it forms an e-family.
Our approach consists of comparing 
the known manifold dimension of $\calW_\kappa(\calY, \calE)$ and that of its e-hull. A dimension mismatch---that is the e-hull forms a higher dimensional manifold---is evidence that $\calW_\kappa(\calY, \calE)$ is curved, while matching dimensions together with additional properties of the lumpable family lead to the conclusion that $\calW_\kappa(\calY, \calE)$ forms an e-family. However, computing the dimension of the e-hull is nontrivial; nonetheless, we provide a basis for it and propose a polynomial-time algorithm to compute its dimension. 
We begin by recalling the definition of the positive lumpable cone
\begin{equation*}
    \calG_\kappa(\calY, \calE) \eqdef \set{ \log F \colon F \in \calF^+_\kappa(\calY, \calE) },
\end{equation*}
and its affine hull in $\calF(\calY, \calE)$,
\begin{equation*} \aff(\calG_\kappa(\calY, \calE)) \eqdef \set{ \sum_{i = 1}^{k} \alpha_i G_i \colon k \in \bbN, \alpha \in \bbR^k, \sum_{i=1}^{k} \alpha_i = 1, G_1, \dots, G_k \in \calG_\kappa(\calY, \calE) }.
\end{equation*}
From Theorem~\ref{theorem:characterization-log-linearity}, we know that $\aff(\calG_\kappa(\calY, \calE)) = \calG_\kappa(\calY, \calE)$---that is $\calG_\kappa(\calY, \calE)$ is an affine space---if and only if $(\calY, \calE)$ has no multi-row merging block with respect to $\kappa$.
For convenience, we define
\begin{equation}
\label{equation:block-type-nomenclature}
\begin{split}
    \calM_{x,x'} &\eqdef \set{ y \in \calS_{x} \colon \exists y_1', y_2' \in \calS_{x'}, y_1' \neq y_2',  (y,y_1'), (y,y_2') \in \calE} \qquad \text{for any $(x,x') \in \calD$},
    \\ \calM &\eqdef \set{ (x,x') \in \calD \colon \calM_{x,x'} \neq \emptyset }, \\
    \calU &\eqdef  \set{(x,x') \in \calD \colon \calM_{x,x'} \neq \calS_{x}}, \\
    \calR &\eqdef \bigcup_{(x_0,x_0') \in \calM } \set{ (y_0, y_0') \in \calM_{x_0, x_0'} \times \calS_{x_0'} \colon (y_0, y_0') \in \calE}.
\end{split}
\end{equation}
In words, $\calM$ is the collection of merging blocks---not necessarily multi-row---with respect to $(\calY, \calE)$ and $\kappa$, for any block $(x,x') \in \calD$, $\calM_{x,x'} \subset \calS_{x}$ is its set of merging rows, $\calU$ is the collection of all blocks containing at least one non-merging row, and $\calR$ is the set of all edges which appear in some merging row. In particular, when $(x,x') \in \calD \setminus \calM$ it holds that $\calM_{x,x'} = \emptyset$.
For each $(x,x') \in \calD$, we fix a set of $\abs{\calS_x}$ anchor edges (see Example~\ref{example:basis-affine-hull-log-lumpable-cone} later in this section),
\begin{equation}
\label{eq:anchor-points}
    E_{x,x'} \eqdef \set{ \left(\bar{y}_1, \bar{y}_1'\right), \dots, \left(\bar{y}_{\abs{\calS_x}}, \bar{y}'_{\abs{\calS_x}}\right) \in \calE \cap \left(\calS_{x} \times \calS_{x'}\right), \text{with } \bar{y}_1, \dots, \bar{y}_{\abs{\calS_x}} \text{ distinct }}.
\end{equation}

\begin{lemma}
\label{lemma:affine-span-log-lumpable-cone-is-linear}
    $\aff(\calG_{\kappa}(\calY, \calE))$ is a linear space.\footnote{In particular, it contains the null vector.}
\end{lemma}

\begin{proof}
By definition $\aff(\calG_{\kappa}(\calY, \calE))$ is an affine space,
    so it suffices to show that it contains the null vector.
It will be convenient to introduce $N \eqdef \max_{x \in \calX} \abs{\calS_{x}}$, and for $y \in \calY, x' \in \calX$,
\begin{equation}
\label{equation:definition-s}
    s_{y,x'} \eqdef \abs{ \left(\set{y} \times \calS_{x'}\right) \cap \calE }.
\end{equation}
For $\alpha \in (0, 1/N)$,
    we let $F_\alpha \in \calF(\calY, \calE)$ be such that for any $(y,y') \in \calE$,
    \begin{equation*}
        F_\alpha(y,y') = \begin{cases}
            \alpha &\text{ when } (y,y') \not \in E_{\kappa(y),\kappa(y')}, \\
            1 - \left( s_{y,\kappa(y')} - 1 \right)\alpha &\text{ otherwise}.
        \end{cases}
    \end{equation*} We similarly introduce $F_\beta$ for $\beta \in (0, 1/N)$ with $\beta \neq \alpha$. 
    By construction, $F_\alpha, F_\beta \in \calF_{\kappa}^{+}(\calY, \calE)$.
    For $t \in \bbR$, let us inspect $G_t = t \log F_\alpha + (1 - t) \log F_\beta$.
    Setting 
    \begin{equation*}
        t = \frac{\log \beta}{\log \left( \frac{\beta}{\alpha} \right)},
    \end{equation*}
we obtain
\begin{equation*}
    G_t(y,y') = \begin{cases}
       0 \qquad \qquad &\text{ when } (y,y') \not \in E_{\kappa(y),\kappa(y')}, \\
       \log \left( 1 - \left(s_{y,\kappa(y')} - 1\right) \beta \right) + \cfrac{\log \beta}{\log \left( \frac{\beta}{\alpha} \right)} \log \left( \cfrac{1 - \left(s_{y,\kappa(y')} - 1\right) \alpha}{1 - \left(s_{y,\kappa(y')} - 1\right) \beta} \right) &\text{ otherwise}.
    \end{cases}
\end{equation*}
Further setting $\alpha = \eta$ and $\beta = 2 \eta$ for $\eta \in (0, 1/(2N))$, we obtain
\begin{equation*}
    G_t(y,y') = \begin{cases}
       0 \qquad \qquad &\text{ when } (y,y') \not \in E_{\kappa(y),\kappa(y')}, \\
       \log \left( 1 - 2 \left(s_{y,\kappa(y')} - 1\right) \eta \right) + \cfrac{\log (2 \eta)}{\log \left( 2 \right)} \log \left( \cfrac{1 - \left(s_{y,\kappa(y')} - 1\right) \eta}{1 - 2 \left(s_{y,\kappa(y')} - 1\right) \eta} \right) &\text{ otherwise}.
    \end{cases}
\end{equation*}
For any fixed $s$, the function 
\begin{equation*}
    h \colon \eta \mapsto \log \left( 1 - 2 (s - 1) \eta \right) + \cfrac{\log (2 \eta)}{\log \left( 2 \right)} \log \left( \cfrac{1 - (s - 1) \eta}{1 - 2 (s - 1) \eta} \right)
\end{equation*}
is continuous, negative, decreasing, and satisfies $\lim_{\eta^+ \to 0} h(\eta) = 0$. So for any $\eps > 0$ there exists $\eta_s$ such that $h(\eta_s) < \eps$.
Finally, taking $s_\star = \min \set{s_{y, \kappa(y')} \colon (y,y') \in E_{x,x'}, (x,x') \in \calD}$, for $\eta < \eta_\star$, we have that
$\nrm{G_t} < \eps \sqrt{\abs{\calE}}$. In other words $0 \in \calF(\calY, \calE)$ can be characterized as an accumulation point of a sequence of elements of $\aff(\calG_\kappa(\calY, \calE))$.
Since $\aff(\calG_{\kappa}(\calY, \calE))$ is an affine subspace of $\calF(\calY, \calE) = \bbR^{\calE}$, it is closed---it contains all its accumulation points, in particular the null function.
\end{proof}

Since they coincide, we henceforth write $\spann(\calG_{\kappa}(\calY, \calE)) = \aff(\calG_{\kappa}(\calY, \calE))$.

\begin{proposition}
For $(x_0,x_0') \in \calU$, let
$G^{\uparrow}_{x_0,x_0'} \in \calF(\calY,\calE)$ be such that for any $(y,y') \in \calE$,
    \begin{equation*}
    G^{\uparrow}_{x_0,x_0'}(y,y') = \pred{ (y,y') \in E_{x_0,x_0'} },
\end{equation*}
where $E_{x_0,x_0'}$ is introduced in \eqref{eq:anchor-points} and the sets $\calU, \calR$ are introduced in \eqref{equation:block-type-nomenclature}.
For any $(y_0, y_0') \in \calR$, let $G^{\uparrow}_{y_0, y_0'} \in \calF(\calY, \calE)$ be such that for any $(y,y') \in \calE$,
\begin{equation*}
    G^{\uparrow}_{y_0, y_0'}(y,y') = \pred{(y,y') = (y_0, y_0')}.
\end{equation*}
The system of functions
\begin{equation*}
    \calB_{\kappa}(\calY, \calE) \eqdef \set{G^{\uparrow}_{x_0,x_0'} \colon (x_0,x_0') \in \calU} \cup \set{G^{\uparrow}_{y_0, y_0'} \colon (y_0,y_0') \in \calR},
\end{equation*}
forms a basis for $\spann(\calG_{\kappa}(\calY, \calE))$, and
\begin{equation*}
    \dim \spann(\calG_\kappa(\calY, \calE)) = 
    \abs{\calU} + \abs{\calR}.
\end{equation*}
\end{proposition}

\begin{proof}
For any $(x_0,x_0') \in \calD$,
and for $b \in \set{0, 1}$,
we let $F^{(b)}_{x_0,x_0'} \in \calF(\calY, \calE)$ be such that for any $(y,y') \in \calE$,
\begin{equation*}
    F^{(b)}_{x_0,x_0'}(y,y') \eqdef \begin{cases}
       \frac{e^{b}}{s_{y,\kappa(y')}} &\text{ when } (\kappa(y),\kappa(y')) = (x_0,x_0'), \\
       \frac{1}{s_{y,\kappa(y')}} &\text{ otherwise},
    \end{cases}
\end{equation*}
where $s_{y, \kappa(y')}$ is defined in \eqref{equation:definition-s}.
By construction, $F^{(b)}_{x_0,x_0'} \in \calF_\kappa^{+}(\calY, \calE)$,
\begin{equation*}
    \log F^{(1)}_{x_0,x_0'} - \log F^{(0)}_{x_0,x_0'}  \in \spann(\calG_{\kappa}(\calY, \calE)),
\end{equation*}
where we relied on the fact that $\aff(\calG_{\kappa}(\calY, \calE)) = \spann(\calG_{\kappa}(\calY, \calE))$ (Lemma~\ref{lemma:affine-span-log-lumpable-cone-is-linear}),
and
when $(x_0, x_0') \not \in \calM$, we have
\begin{equation*}
    \log F^{(1)}_{x_0,x_0'} - \log F^{(0)}_{x_0,x_0'} = G^{\uparrow}_{x_0,x_0'}.
\end{equation*}
Next, for $(y_0, y_0') \in \calR$, note that $s_{y_0, \kappa(y_0')} > 1$, and let us construct
\begin{equation*}
    \widetilde{F}^{\uparrow}_{y_0, y'_0}(y,y') \eqdef \begin{cases}
        \frac{1}{2} &\text{ when } (y, y') = (y_0, y_0'), \\
        \frac{1}{2 \left(s_{y_0,\kappa(y'_0)} - 1\right)} &\text{ when } y = y_0, \kappa(y') = \kappa(y'_0),  y' \neq y_0', \\
       \frac{1}{s_{y,\kappa(y')}} &\text{ otherwise}.
    \end{cases}
\end{equation*}
Observe that
\begin{equation*}
    \left(\log \widetilde{F}^{\uparrow}_{y_0, y'_0} - \log F^{(0)}_{\kappa(y_0), \kappa(y_0')} \right)(y,y') = \begin{cases}
       \log\left(\frac{s_{y_0,\kappa(y'_0)}}{2}\right)   &\text{ when } (y, y') = (y_0, y_0'), \\
        \log\left(\frac{s_{y_0,\kappa(y'_0)}}{2\left(s_{y_0,\kappa(y'_0)} - 1\right)}\right)  &\text{ when } y = y_0, \kappa(y') = \kappa(y'_0),  y' \neq y_0', \\
       0 &\text{ otherwise}.
    \end{cases}
\end{equation*}
In order to construct $G_{y_0,y_0'}^{\uparrow}$ we first introduce
\begin{equation*}
\begin{split}
    G^{\uparrow}_{y_0, \kappa(y_0')} &\eqdef 
    \frac{\sum_{\bar{y}_0' \in \calS_{\kappa(y_0')}\colon (y_0, \bar{y}_0') \in \calE} \left(\log \widetilde{F}^{\uparrow}_{y_0, \bar{y}'_0} - \log F^{(0)}_{\kappa(y_0),\kappa(y_0')} \right)}{\log\left(s_{y_0, \kappa(y_0')}\right) - \log(2) + \left(s_{y_0, \kappa(y_0')} - 1\right) \left(\log\left(s_{y_0, \kappa(y_0')}\right)- \log(2) - \log\left(s_{y_0, \kappa(y_0')} - 1\right)\right)}.\\
    \end{split}
    \end{equation*}
    Observe that
    \begin{equation*}
         G^{\uparrow}_{y_0, \kappa(y_0')}(y,y') = \begin{cases}
       1  &\text{ when } y = y_0 \text{ and } \kappa(y') = \kappa(y_0'), \\
       0 &\text{ otherwise}.
    \end{cases}
    \end{equation*}
    We then construct $G_{y_0,y_0'}^{\uparrow}$ as
    \begin{equation*}
        \begin{split}
    G_{y_0,y_0'}^{\uparrow} &= \left(1 - \frac{ \log\left(s_{y_0, \kappa(y_0')} - 1\right)}{\log\left( s_{y_0, \kappa(y_0')}\right) - \log(2)} \right) \left( \frac{\log \widetilde{F}^{\uparrow}_{y_0, y'_0} - \log F^{(0)}_{\kappa(y_0),\kappa(y_0')}}{\log\left(s_{y_0, \kappa(y_0')}\right) - \log(2) - \log\left(s_{y_0, \kappa(y_0')} - 1\right)} - G^{\uparrow}_{y_0, \kappa(y_0')} \right).
    \end{split}
\end{equation*}
We obtain that for $(y,y') \in \calE$,
\begin{equation*}
\begin{split}
    G^{\uparrow}_{y_0, y_0'}(y,y') &= \begin{cases}
       1  &\text{ when } (y,y') = (y_0, y_0'), \\
       0 &\text{ otherwise}.
    \end{cases}
\end{split}
\end{equation*}
The construction of $G^{\uparrow}_{x_0,x_0'}$ for $(x_0,x_0') \in \calU$ follows,
\begin{equation*}
    G^{\uparrow}_{x_0,x_0'} = \log F^{(1)}_{x_0,x_0'} - \log F^{(0)}_{x_0,x_0'} - \sum_{(y_0, y_0') \in \calR \cap \left( \calS_{x_0} \times \calS_{x_0'} \right)} G^{\uparrow}_{y_0, y_0'}.
\end{equation*}
Generativity and linear independence of the system of functions are immediate.
\end{proof}

\begin{example}
\label{example:basis-affine-hull-log-lumpable-cone}
    Consider
    \begin{equation*}
    \calW_{\kappa}(\calY, \calE) \sim \left(\begin{tblr}{c|[dashed]ccc}
  \oplus & \oplus & \plus & \zero  \\\hline[dashed]
  \oplus & \oplus & \plus & \plus \\
  \oplus & \zero & \oplus & \zero \\
  \oplus & \zero & \zero & \oplus \\
\end{tblr}\right),
\end{equation*}
 where $\oplus$ indicates that the edge is fixed as an anchor---that is it belongs to $E_{x,x'}$ for some $(x,x') \in \calD$.
A basis for $\spann(\calG_\kappa(\calY, \calE))$ is then given by
    \begin{equation*}
    \begin{split}
    \left(\begin{tblr}{c|[dashed]ccc}
  1 & \zero & \zero & \zero  \\\hline[dashed]
  \zero & \zero & \zero & \zero \\
  \zero & \zero & \zero & \zero \\
  \zero & \zero & \zero & \zero \\
\end{tblr}\right), \left(\begin{tblr}{c|[dashed]ccc}
  \zero & 1 & \zero & \zero  \\\hline[dashed]
  \zero & \zero & \zero & \zero \\
  \zero & \zero & \zero & \zero \\
  \zero & \zero & \zero & \zero \\
\end{tblr}\right),\left(\begin{tblr}{c|[dashed]ccc}
  \zero & \zero & 1 & \zero  \\\hline[dashed]
  \zero & \zero & \zero & \zero \\
  \zero & \zero & \zero & \zero \\
  \zero & \zero & \zero & \zero \\
\end{tblr}\right), \left(\begin{tblr}{c|[dashed]ccc}
  \zero & \zero & \zero & \zero  \\\hline[dashed]
  1 & \zero & \zero & \zero \\
  1 & \zero & \zero & \zero \\
  1 & \zero & \zero & \zero \\
\end{tblr}\right),
\end{split}
\end{equation*}
\begin{equation*}
    \begin{split}
 \left(\begin{tblr}{c|[dashed]ccc}
  \zero & \zero & \zero & \zero  \\\hline[dashed]
  \zero & 1 & \zero & \zero \\
  \zero & \zero & 1 & \zero \\
  \zero & \zero & \zero & 1 \\
\end{tblr}\right), \left(\begin{tblr}{c|[dashed]ccc}
  \zero & \zero & \zero & \zero  \\\hline[dashed]
  \zero & 1 & \zero & \zero \\
  \zero & \zero & \zero & \zero \\
  \zero & \zero & \zero & \zero \\
\end{tblr}\right), \left(\begin{tblr}{c|[dashed]ccc}
  \zero & \zero & \zero & \zero  \\\hline[dashed]
  \zero & \zero & 1 & \zero \\
  \zero & \zero & \zero & \zero \\
  \zero & \zero & \zero & \zero \\
\end{tblr}\right), \left(\begin{tblr}{c|[dashed]ccc}
  \zero & \zero & \zero & \zero  \\\hline[dashed]
  \zero & \zero & \zero & 1 \\
  \zero & \zero & \zero & \zero \\
  \zero & \zero & \zero & \zero \\
\end{tblr}\right).
\end{split}
\end{equation*}
\end{example}

\begin{example}
    Consider
    \begin{equation*}
    \calW_{\kappa}(\calY, \calE) \sim \left(\begin{tblr}{cc|[dashed]cc}
  \oplus & \zero & \oplus & \zero \\
  \zero & \oplus & \zero & \oplus \\\hline[dashed]
  \zero & \oplus & \oplus & \plus \\
  \oplus & \zero & \zero & \oplus \\
\end{tblr}\right).
\end{equation*}
A basis for $\spann(\calG_{\kappa}(\calY, \calE))$ is then given by
\begin{equation*}
\begin{split}
    \left(\begin{tblr}{cc|[dashed]cc}
  1 & \zero & \zero & \zero \\
  \zero & 1 & \zero & \zero \\\hline[dashed]
  \zero & \zero & \zero & \zero \\
  \zero & \zero & \zero & \zero \\
\end{tblr}\right), \left(\begin{tblr}{cc|[dashed]cc}
  \zero & \zero & 1 & \zero \\
  \zero & \zero & \zero & 1 \\\hline[dashed]
  \zero & \zero & \zero & \zero \\
  \zero & \zero & \zero & \zero \\
\end{tblr}\right),\left(\begin{tblr}{cc|[dashed]cc}
  \zero & \zero & \zero & \zero \\
  \zero & \zero & \zero & \zero \\\hline[dashed]
  \zero & 1 & \zero & \zero \\
  1 & \zero & \zero & \zero \\
\end{tblr}\right),\\ 
\left(\begin{tblr}{cc|[dashed]cc}
  \zero & \zero & \zero & \zero \\
  \zero & \zero & \zero & \zero \\\hline[dashed]
  \zero & \zero & 1 & \zero \\
  \zero & \zero & \zero & 1 \\
\end{tblr}\right),\left(\begin{tblr}{cc|[dashed]cc}
  \zero & \zero & \zero & \zero \\
  \zero & \zero & \zero & \zero \\\hline[dashed]
  \zero & \zero & 1 & \zero \\
  \zero & \zero & \zero & \zero \\
\end{tblr}\right), \left(\begin{tblr}{cc|[dashed]cc}
  \zero & \zero & \zero & \zero \\
  \zero & \zero & \zero & \zero \\\hline[dashed]
  \zero & \zero & \zero & 1 \\
  \zero & \zero & \zero & \zero \\
\end{tblr}\right).
\end{split}
\end{equation*}

\end{example}

We recall the definition of the exponential hull of a family of stochastic matrices $\calV \subset \calW(\calY, \calE)$ as the smallest exponential family which contains $\calV$.
\begin{definition}[Exponential hull, {\citealp[Definition~7]{wolfer2021information}}]
\label{definition:e-hull}\label{nom:e-hull}
Let $\calV$ be a sub-family of $\calW(\calY, \calE)$.
\begin{equation*}
\begin{split}
\ehull(\mathcal{V}) \eqdef \Bigg\{ \stoch(\widetilde{P}) &\colon \log \widetilde{P} = \sum_{i = 1}^{k} \alpha_i \log P_i, k \in \bbN, \alpha \in \bbR^k, \sum_{i = 1}^{k} \alpha_i = 1, P_1, \dots P_k \in \calV \Bigg\},    
\end{split}
\end{equation*}
where $\stoch$-normalization was introduced in Definition~\ref{definition:s-normalization}.

\end{definition} 
In particular, $\ehull(\calV) = \calV$ if and only if $\calV$ forms an e-family.
 For instance, it is known that the exponential hull of symmetric stochastic matrices yields the reversible family \citep[Theorem~9]{wolfer2021information}.
In our analysis, it will be convenient to further define $\calH_\kappa(\calY, \calE)$ and $\overline{\calH}_\kappa(\calY, \calE)$ as follows.
\begin{equation*}
\begin{split}
\calH_\kappa(\calY, \calE) &\eqdef \set{ \log P \colon P \in \calW_{\kappa}(\calY, \calE)} \subset \calG_{\kappa}(\calY, \calE), \\
\overline{\calH}_\kappa(\calY, \calE) &\eqdef \set{ \log P \colon P \in \ehull(\calW_{\kappa}(\calY, \calE))}.
    \end{split}
\end{equation*}
Observe that $\overline{\calH}_\kappa(\calY, \calE)$ is isomorphic to the affine hull of $\calH_\kappa(\calY, \calE)$ in $\calF(\calY, \calE)/\calN(\calY, \calE)$.

\begin{proposition}
\label{proposition:log-exponential-hull-isomorphic-to-quotient}
    \begin{equation*}
       \overline{\calH}_\kappa(\calY, \calE) \cong \left(\spann(\calG_\kappa(\calY, \calE)) \oplus \calN(\calY, \calE) \right) / \calN(\calY, \calE).
    \end{equation*}
\end{proposition}

\begin{proof}
    Let $\overline{H} \in \overline{\calH}_\kappa(\calY, \calE)$. There exists $\overline{G} \in \spann(\calG_{\kappa}(\calY, \calE))$ such that
    \begin{equation*}
    \exp\left(\overline{H}\right) = \stoch\left( \exp \left(\overline{G}\right) \right).
    \end{equation*}
It follows that we can write
\begin{equation*}
    \overline{H} = \overline{G} + N,
\end{equation*}
for some $N \in \calN(\calY, \calE)$. As a result,
\begin{equation*}
    \overline{H} \in \left(\spann(\calG_\kappa(\calY, \calE)) \oplus \calN(\calY, \calE) \right) / \calN(\calY, \calE).
\end{equation*}
    Conversely, let $\overline{G} \in \spann(\calG_{\kappa}(\calY, \calE))$ and $N \in \calN(\calY, \calE)$. It holds that
    \begin{equation*}
        \stoch \left( \exp \left(\overline{G} + N \right)  \right) \in \ehull(\calW_{\kappa}(\calY, \calE)),
    \end{equation*}
    hence
    \begin{equation*}
        \left[\overline{G} + N \right]_{\calN(\calY, \calE)} \in \overline{\calH}_\kappa(\calY, \calE).
    \end{equation*}
\end{proof}

Our main theorem will rely on the following two properties of topological and smooth manifolds.
\begin{proposition}[{\cite[Proposition~4.1]{lee2010introduction}}]
\label{proposition:open-closed-argument}
     A topological space $\calM$ is connected if and only if the only subsets
of $\calM$ that are both open and closed in $\calM$ are $\emptyset$ and $\calM$ itself.
\end{proposition}

\begin{proposition}[{\cite[Proposition~5.1]{lee2013smooth}}]
\label{proposition:open-submanifolds}
Suppose $\calM$ is a smooth manifold. The embedded submanifolds of codimension 0 in $\calM$ are exactly the open submanifolds.  
\end{proposition}

\begin{theorem}[General dimension criterion for e-families]
\label{theorem:general-dimensional-criterion}
    Let $\calV$ be a submanifold (without boundary) of $\calW(\calY, \calE)$ such that
    \begin{enumerate}
        \item[\enumi] $\dim \calV = \dim \ehull(\calV)$.
        \item[\enumii] There exists some finite dimensional (Hausdorff) topological vector space $\calU$, a linear map 
        $$\phi \colon \calF(\calY, \calE) \to \calU,$$ 
        and $C \in \calU$ such that 
        $$\calV = \calW(\calY, \calE) \cap \calL,$$ 
    \end{enumerate}
    where 
    \begin{equation*}
        \calL \eqdef \set{ F \in \calF(\calY, \calE), \phi(F) = C }.
    \end{equation*}
    Then $\calV$ forms an e-family.
\end{theorem}

\begin{proof}
    Here, we regard $\calV$ as a smooth submanifold of $\ehull (\calV)$, endowed with the subspace topology induced from $\calF(\calY, \calE) \cong \bbR^\calE$.
    Our strategy is to show that under the assumptions of the theorem $\calV = \ehull(\calV)$, that is $\calV$ forms an e-family. 
    Since $\calL$ is defined by a set of linear equality constraints it must be convex. Indeed,
    let $F_0, F_1 \in \calL$, then $\phi(F_0) = C$ and $\phi(F_1) = C$.
    Then, for any $t \in [0,1]$, $F_t = (1 - t) F_0 + t F_1$ satisfies
    \begin{equation*}
        \phi(F_t) = \phi((1 - t) F_0 + t F_1) = (1 - t) \phi(F_0) + t \phi(F_1) = (1 - t) C + t C = C,
    \end{equation*}
    thus $\calL$ is convex. Similarly, $\calW(\calY, \calE)$ is convex and so is the intersection $\calW(\calY, \calE) \cap \calL$.
    As a result of convexity, $\calV$ is path-connected. We can therefore rely on an open-closed argument (Proposition~\ref{proposition:open-closed-argument}). Namely, if we can show that $\calV$ is clopen, our claim will follow. 
    \begin{enumerate}[$\diamond$]
        \item \textbf{$\calV$ is closed in $\ehull(\calV)$:}
        Since $\calV \subset \ehull(\calV)$ and $\calV \subset \calL$,
        we can write
        \begin{equation*}
            \calV = \calW(\calY, \calE) \cap \calL = \ehull(\calV) \cap \calL.
        \end{equation*}
        Since $\calL$ is defined by a finite number of linear equality constraints, it is a closed set. Indeed, since $\phi$ is a linear map on a finite dimensional space, it is continuous. Moreover, the preimage of a closed set under a continuous function is closed. Since $\calU$ is Hausdorff, the singleton $\set{C}$ is closed, and it follows that $\calL$ is closed in $\calF(\calY, \calE)$. 
        Hence $\calV$ is closed in $\ehull(\calV)$ by characterization of closed sets in the subspace topology.
        \item \textbf{$\calV$ is open in $\ehull(\calV)$:} Since by assumption, $\calV$ is a smooth submanifold (without boundary) of $\ehull(\calV)$ and the co-dimension of $\calV$ in $\ehull(\calV)$ is $0$, it must be from Proposition~\ref{proposition:open-submanifolds} that $\calV$ is open.
    \end{enumerate}
    Since $\calV$ is clopen in $\ehull(\calV)$, it must be that $\calV = \ehull(\calV)$ and $\calV$ forms an e-family.
\end{proof}

\begin{theorem}[Dimensional criterion for the $\kappa$-lumpable family]
\label{theorem:dimensional-criterion}
   $\calW_{\kappa}(\calY, \calE)$ forms an e-family if and only if 
   \begin{equation*}
        \dim \left(\spann\left(\calG_\kappa(\calY, \calE)\right) \oplus \calN(\calY, \calE) \right)
        = \abs{\calE} + \abs{\calY} + \abs{\calD} - \abs{\calX} - \sum_{(x,x') \in \calD} \abs{\calS_x}.
   \end{equation*}
\end{theorem}

\begin{proof} 

From the matrix form of the Kemeny--Snell condition for lumpability---refer to\eqref{eq:kemeny-snell-condition-matrix-form}---we can write that
\begin{equation*}
    \calW_\kappa(\calY, \calE) = \calW(\calY,\calE) \cap \set{ F \in \calF(\calY, \calE) \colon \phi(F) = C },
\end{equation*}
where $\phi \colon \calF(\calY, \calE) \to \bbR^{\abs{\calY} \times \abs{\calX}}$ is the linear map defined for $F \in \calF(\calY, \calE)$ by
\begin{equation*}
    \phi(F) = (K \overline{K} \trn - I) F K \qquad \text{ and } \qquad C = 0_{\abs{\calY} \times \abs{\calX}},
\end{equation*}
and where $K$ and $\overline{K}$ are defined in Section~\ref{section:lumpability}.
Finally, $\calW_\kappa(\calY,\calE)$ is a manifold without (manifold) boundary since, for every point, 
we can take a coordinate chart that is diffeomorphic to an open set in $\bbR^d$; for instance, we can take a coordinate
that is induced by a coordinate of $\calW(\calX,\calD)$ and a coordinate of  embeddings.
Finally, from Proposition~\ref{proposition:log-exponential-hull-isomorphic-to-quotient}, the dimension condition of the theorem corresponds to the first condition of Theorem~\ref{theorem:general-dimensional-criterion}.
The claim therefore follows from an application of Theorem~\ref{theorem:general-dimensional-criterion}.
\end{proof}

A basis for $\calN(\calY, \calE)$ can be constructed as follows.
For $y_0 \in \calY$, we define $N_{y_0} \in \calF(\calY, \calE)$ be such that for any $(y,y') \in \calE$,
\begin{equation*}
    N_{y_0}(y,y') = \pred{y' = y_0} - \pred{y = y_0}.
\end{equation*}
Let $y_\star \in \calY$ be arbitrary; then 
\begin{equation*}
    \set{C} \cup  \set{ N_{y_0} \colon y_0 \in \calY \setminus \set{y_\star}},
\end{equation*}
where $C \equiv 1$ is the constant unit function over $\calE$,
forms a basis for $\calN(\calY, \calE)$. As a consequence, the bases for $\spann(\calG_\kappa(\calY, \calE))$ and $\calN(\calY, \calE)$ can be concatenated, and determining the rank of the resulting family can be obtained algorithmically, for instance with Gaussian elimination. Before performing this somewhat costly computation, one can also proceed with the following preliminary verification which does not require computing the rank of a family of functions.

\begin{corollary}
\label{corollary:dimensional-criterion-simplified}
If 
    \begin{equation*}
        \underbrace{\sum_{(x,x') \in \calD} \abs{\calS_x}}_{ \text{\# of rows in every block}} > \underbrace{\left( \abs{\calD} - \abs{\calU} \right)}_{  \substack{\text{\# completely} \\ \text{merging blocks}} } + \underbrace{\left(\abs{\calY} - \abs{\calX}\right)}_{\substack{\text{state space} \\ \text{compression}}} + \underbrace{\left(\abs{\calE} - \abs{\calR} \right)}_{  \substack{\text{\# of non-merging}\\ \text{transitions}}},
    \end{equation*}
    where $\calU$ and $\calR$ are defined in \eqref{equation:block-type-nomenclature},
    then $\calW_{\kappa}(\calY, \calE)$ does not form an e-family.
\end{corollary}
\begin{proof}
The claim directly follows from the inequality,
    \begin{equation*}
        \dim \spann\left(\calG_\kappa(\calY, \calE)\right) \leq \dim \left(\spann\left(\calG_\kappa(\calY, \calE)\right) \oplus \calN(\calY, \calE) \right).
    \end{equation*}
\end{proof}

\begin{example}[Example~{\ref{example:basis-affine-hull-log-lumpable-cone}} continued]
Recall the family
    \begin{equation*}
    \calW_{\kappa}(\calY, \calE) \sim \left(\begin{tblr}{c|[dashed]ccc}
  \oplus & \oplus & \plus & \zero  \\\hline[dashed]
  \oplus & \oplus & \plus & \plus \\
  \oplus & \zero & \oplus & \zero \\
  \oplus & \zero & \zero & \oplus \\
\end{tblr}\right).
\end{equation*}
A basis of functions $\calB_\kappa(\calY, \calE)$ for $\calG_{\kappa}(\calY, \calE)$ is given in Example~\ref{example:basis-affine-hull-log-lumpable-cone}.
The space $\calN(\calY, \calE)$ is spanned by 
    \begin{equation*}
    \left(\begin{tblr}{c|[dashed]ccc}
  1 & 1 & 1 & 0 \\\hline[dashed]
  1 & 1 & 1 & 1 \\
  1 & \zero & 1 & \zero \\
  1 & \zero & \zero & 1 \\
\end{tblr}\right), \left(\begin{tblr}{c|[dashed]ccc}
  \zero & 1 & \zero & \zero \\\hline[dashed]
  -1 & \zero & -1 & -1 \\
  \zero & \zero & \zero & \zero \\
  \zero & \zero & \zero & \zero \\
\end{tblr}\right), \left(\begin{tblr}{c|[dashed]ccc}
  \zero & \zero & 1 & \zero \\\hline[dashed]
  \zero & \zero & 1 & \zero \\
  -1 & \zero & \zero & \zero \\
  \zero & \zero & \zero & \zero \\
\end{tblr}\right), \left(\begin{tblr}{c|[dashed]ccc}
  \zero & \zero & \zero & \zero \\\hline[dashed]
  \zero & \zero & \zero & 1 \\
  \zero & \zero & \zero & \zero \\
  -1 & \zero & \zero & \zero \\
\end{tblr}\right). 
\end{equation*}
On one end, the manifold dimension of the lumpable family is given by
\begin{equation*}
    \dim \calW_{\kappa}(\calY, \calE) = \abs{\calE} + \abs{\calD} - \abs{\calX} - \sum_{(x,x') \in \calD} \abs{\calS_{x}} = 11 + 4 - 2 - 8 = 5.
\end{equation*}
On the other hand, concatenating the two bases, one can algorithmically verify that
\begin{equation*}
\begin{split}
    \dim \left(\spann\left(\calG_\kappa(\calY, \calE)\right) \oplus \calN(\calY, \calE) \right) &= 10.
\end{split}
\end{equation*}
Consequently,
\begin{equation*}
    \dim \left(\spann\left(\calG_\kappa(\calY, \calE)\right) \oplus \calN(\calY, \calE) \right) > \dim \calW_{\kappa}(\calY, \calE) + \dim \calN(\calY, \calE),
\end{equation*}
and from Theorem~\ref{theorem:dimensional-criterion}, $\calW_{\kappa}(\calY, \calE)$ does not form an e-family.
\end{example}

\begin{remark}
\label{remark:comparison-combinatorial-dimensional-criteria}
The astute reader may notice that our characterization of e-families of lumpable families based on a dimension argument is strictly stronger than our necessary and sufficient combinatorial conditions, and thus question the need for the latter. However, we argue that criteria based on purely combinatorial properties of the connection graph are valuable from at least two perspectives. First and foremost, from an algorithmic complexity standpoint (refer to Section~\ref{section:algorithmics}), combinatorial criteria can be verified much more efficiently. In fact, for small state spaces, these arguments can even be confirmed by the human eye or with pen and paper. Second, combinatorial arguments can already successfully explain the nature of ``most'' families, bypassing the need to invoke the more involved dimensional argument. A heuristic argument to support our claim consists in observing that over large state spaces---where the difference in algorithmic complexity between dimensional and combinatorial criteria is the most relevant---a large proportion of irreducible digraphs are in fact two-connected, and thus we expect most merging blocks to be redundant. As a result, we can propose a layered argument, where most cases can be treated combinatorially, while the remaining cases are treated with the dimension argument.
\end{remark}

\section{Monotonicity and stability}
\label{section:monotonicity}
For edge sets $\calE, \calE' \subset \calY^2$ such that $\calE \subset \calE'$, it is clear that when $(\calY, \calE)$ is strongly connected, so must be $(\calY, \calE')$, and the graph $(\calY, \calE)$ can be incrementally transformed into $(\calY, \calE')$ by adding individual edges, naturally leading to the construction of a monotonous sequence of consecutive elements $\calE = \calE_0 \subsetneq \calE_1 \subsetneq \dots \subsetneq \calE_{L} = \calE'$, where $L = \abs{\calE'} -  \abs{\calE}$.
When considering lumpability,
more care is required. Indeed, even if $\calW_\kappa(\calY, \calE)$ is non-vacuous, adding a single edge can easily disrupt lumpability, leading to $\calW_\kappa(\calY, \calE') = \emptyset$.
Only by adding edges in already existing blocks, or introducing new lumpable blocks can we guarantee the existence of lumpable chains under the new graph.
We first explain how to construct a monotonous sequence of essentially consecutive families of irreducible and lumpable stochastic matrices.

\begin{lemma}[Chaining]
\label{lemma:chaining-edge-sets}
Let $\calE, \calE' \in \calY^2$ be such that $\calE \subset \calE'$, and $\calW_{\kappa}(\calY, \calE)$ and $\calW_{\kappa}(\calY, \calE')$ are non-vacuous families of lumpable stochastic matrices. Then,
 there exists $L \in \bbN$ and a finite sequence of edge sets $\calE_0, \calE_1, \dots, \calE_L \subset \calY^2$, such that
 \begin{enumerate}
     \item[\enumi] $\calE_0 = \calE$, $\calE_L = \calE'$.
     \item[\enumii] The sequence $\{\calE_\ell \}_{\ell = 0, \dots, L }$ is strictly monotone increasing, that is, for any $\ell \in \set{0, \dots, L - 1}$, $\calE_\ell \subsetneq \calE_{\ell+1}$.
     \item[\enumiii] The sequence $\{\calE_\ell \}_{\ell = 0, \dots, L }$ is consecutive, in the sense where for any $\ell \in \set{0, \dots, L -1}$, it holds that for any $\calE'' \subset \calY^2$ such that $\calE_\ell \subset \calE'' \subset \calE_{\ell+1}$ either $\calE'' = \calE_\ell, \calE'' = \calE_{\ell+1}$ or $\calW_{\kappa}(\calY ,\calE'') = \emptyset$.
     \item[\enumiv] For any $\ell \in \set{0, \dots, L}$, $\calW_\kappa(\calY, \calE_\ell)$ is non-vacuous.
     \item[\enumv] For any $\ell \in \set{0, 1\dots, L - 1}$, $\calE_\ell$ and $\calE_{\ell+1}$ only differ either on an edge-link or a block-link:
     \begin{description}
         \item[Edge-link:] We say that $(y_\star, y_\star') \in \calE_{\ell+1}$ is an edge-link between $\calE_{\ell}$ and $\calE_{\ell+1}$ whenever             $$\kappa_2(\calE_{\ell+1}) = \kappa_2(\calE_{\ell}), \text{ and } \calE_{\ell+1} \setminus \calE_{\ell} = \set{(y_\star, y_\star')}.$$
 \item[Block-link:] We say that $(x_\star, x'_\star) \in \kappa_2(\calE_{\ell+1})$ is a block-link between $\calE_\ell$ and $\calE_{\ell+1}$ whenever
 $$\kappa_2(\calE_{\ell+1}) \neq \kappa_2(\calE_{\ell}), \calE_{\ell+1} \setminus \calE_{\ell} \subset \calS_{x_\star} \times \calS_{x_\star'} \text{, and } (x_\star, x'_\star) \text{ is non-merging}.$$
     \end{description}
 \end{enumerate}
\end{lemma}
\begin{proof}
    Let us set $\calE_0 = \calE$ and let us denote $$\kappa_2(\calE') \setminus \kappa_2(\calE_0) = \set{(x_1,x'_1), \dots , (x_B, x'_B)},$$
    the indexed collection of $B \in \bbN$ non-zero blocks in lumped $\calE'$ that vanish in lumped $\calE$.
    For such a block $b \in [B]$, we denote $s_b = \abs{\calS_{x_b}}$, and we let
    \begin{equation*}
         (y_{b,1}, y'_{b,1}), \dots, (y_{b, s_b}, y'_{b, s_b}) \in \calS_{x_b} \times \calS_{x'_b} \cap \calE'
    \end{equation*}
    be a collection of $s_b$ pairs such that 
    $y_{b,1}, y_{b,2}, \dots, y_{b,s_b}$ are all distinct.
    Note that such a collection of $s_b$ pairs must necessarily exist from our assumption that $\calW_\kappa(\calY, \calE') \neq \emptyset$.
    For simplicity, we write
    \begin{equation*}
        \Delta \calE_b \eqdef \set{ (y_{b,1}, y'_{b,1}), \dots, (y_{b, s_b}, y'_{b, s_b}) },
    \end{equation*}
    and for $b \in [B]$, we then define inductively
    $$\calE_{b} = \calE_{b - 1} \cup \Delta \calE_b.$$
    Observe that by construction, $\calW_\kappa(\calY, \calE_b) \neq \emptyset$ for any $b = 0, \dots, B$, and that by exhaustion, we obtain $\kappa_2(\calE') = \kappa_2(\calE_{B})$.
    As a second step, we now denote
    \begin{equation*}
        \calE' \setminus \calE_B = \set{ (y_1, y'_1), \dots, (y_{E}, y'_{E}) }
    \end{equation*}
    the set of edges in $\calE'$ which are still missing in $\calE_B$.
    For $e \in [E]$, defining inductively
    \begin{equation*}
        \calE_{B + e} = \calE_{B + e - 1} \cup \set{ (y_e, y'_e) },
    \end{equation*}
    we obtain a full chain of non-vacuous lumpable families
    \begin{equation*}
        \calE_0 \subsetneq \dots \subsetneq \calE_{B} \subsetneq \calE_{B +1 } \subsetneq \dots \subsetneq \calE_{B + E},
    \end{equation*}
    which concludes the claim with $L = B + E$.
\end{proof}

\begin{remark}
    Links offer minimal updates in as much as it is not possible to insert additional (non-vacuous) elements in between links.
    Moreover, while chains are not unique, observe that they all share the same length
    \begin{equation*}
        L = \sum_{(x,x') \in \kappa_{2}(\calE)} \abs{ (\calE' \setminus \calE) \cap (\calS_{x} \times \calS_{x'}) } + \sum_{(x,x') \in \kappa_2(\calE') \setminus \kappa_2(\calE)} \abs{ \calE' \cap (\calS_{x} \times \calS_{x'})} - \abs{\calS_{x}} + 1.
    \end{equation*}
\end{remark}

We now state and show a monotonicity property of e-families.

\begin{theorem}[Monotonicity]
\label{theorem:monotonicity}
Let two edge sets $\calE \subset \calE' \subset \calY^2$
be such that both $\calW_\kappa(\calY, \calE)$ and $\calW_\kappa(\calY, \calE')$ are non-vacuous. If $\calW_\kappa(\calY, \calE')$ forms an e-family, then $\calW_\kappa(\calY, \calE)$ forms an e-family.
\end{theorem}

\begin{proof}
To avoid trivialities, we assume that $\kappa$ is non-degenerate (refer to Remark~\ref{remark:degenerate-lumping}) and that $\abs{\calY} \geq 3$.
    Consider the max-norm on $\calF(\calY, \calY^2)$ defined by $\nrm{A} = \max_{y,y' \in \calY} \abs{A(y,y')}$.
    We prove the contrapositive of the claim.
    Let us suppose that $\calW_\kappa(\calY, \calE)$ does not form an e-family.
    In the case where $\calE' = \calY^2$, the claim
that $\calW_\kappa(\calY, \calE')$ does not form an e-family follows immediately from Corollary~\ref{corollary:complete-graph}. From now on, we can therefore assume that $\calE' \neq \calY^2$.
    From \citet[Corollary~3]{nagaoka2005exponential}, there exist $P_0, P_1 \in \calW_\kappa(\calY, \calE)$ and $t_\star \in \bbR_+$ such that $\gamma^{(e)}_{P_0 P_1}(t_\star) \not \in \calW_\kappa(\calY, \calE)$. 
    Treating $\calW_\kappa(\calY, \calE)$, $\calW(\calY, \calE)$ and $\calF_\kappa(\calY, \calE')$ as subsets of $\calF(\calY, \calY^2)$, and from our assumption that $\calW_\kappa(\calY, \calE') \neq \emptyset$, we can write\footnote{Recall that a function in $\calF_\kappa(\calY, \calE')$ can still vanish on $\calE'$.}
    \begin{equation*}
        \calW_\kappa(\calY, \calE) = \calW(\calY, \calE) \cap \calF_\kappa(\calY, \calE').
    \end{equation*}
    Since $\calW(\calY, \calE)$ is an e-family, it holds that $\gamma^{(e)}_{P_0, P_1} \subset \calW(\calY, \calE)$ and necessarily 
$\gamma^{(e)}_{P_0, P_1}(t_\star) \not \in \calF_\kappa(\calY, \calE').$
Since $\calE' \neq \calY^2$, it holds that $\dim \calF_\kappa(\calY, \calE') < \dim \calF(\calY, \calY^2)$.
Being a strict linear subspace of $\calF(\calY, \calY^2)$, $\calF_\kappa(\calY, \calE')$ has therefore no interior points and the boundary set is given by $\partial \calF_\kappa(\calY, \calE') = \calF_\kappa(\calY, \calE')$.
As a consequence, there must exist $\zeta \in \bbR_+$ such that the curve at parameter $t_\star$ satisfies 
    \begin{equation}
    \label{eq:parameter-time-outside}
        \inf_{F \in \calF_\kappa(\calY, \calE')} \nrm{F - \gamma^{(e)}_{P_0, P_1}(t_\star)} > \zeta.
    \end{equation}
    We first treat the two simple cases where $\calE$ and $\calE'$ are connected by either an edge-link or a block-link.
    From the chaining Lemma~\ref{lemma:chaining-edge-sets}, we will then deduce the general claim.

    \paragraph{Edge-link case:}
    Let us first assume that $\kappa_2(\calE) = \kappa_2(\calE')$ and that $\calE$ and $\calE'$ differ on a single edge
    $(y_\star, y_\star') \in \calE' \setminus \calE$, where $(y_\star, y_\star') \in \calS_{x_\star} \times \calS_{x'_\star}$. By construction, since $\calW_\kappa(\calY, \calE)$ is non-vacuous, there exists $\bar{y}'_\star \in \calS_{x'_\star}$ with $\bar{y}'_\star \neq y'_\star$ such that $(y_\star, \bar{y}'_\star) \in \calE$.
    Let $\eta \in (0, \bar{\eta}/2]$ with $\bar{\eta} \eqdef \min \{ P_0(y_\star, \bar{y}'_\star), P_1(y_\star, \bar{y}'_\star) \}$, and for $k \in \{ 0, 1\}$ and $(y,y') \in \calY^2$ set
    \begin{equation*}
    \begin{split}
        P'_k(y, y') &= \begin{cases} 0 &\text{ when } (y,y') \not \in \calE' \\
           \eta &\text{ when } (y,y') = (y_\star, y'_\star) \\
           P_k(y, y') - \eta &\text{ when } (y,y') = (y_\star, \bar{y}'_\star) \\
           P_k(y, y') &\text{ otherwise.}
        \end{cases}
        \end{split}
    \end{equation*}
    The resulting $P'_0$ an $P'_1$ are irreducible and $\kappa$-lumpable stochastic matrices, that is $P'_0, P'_1 \in \calW_{\kappa}(\calY, \calE')$.
    We proceed to inspect the curve $\gamma^{(e)}_{P'_0, P'_1} \colon \bbR \to \calW(\calY, \calE')$ and for simplicity, denote $P'_t \eqdef \gamma^{(e)}_{P'_0, P'_1}(t)$ the transition matrix such that for any $(y,y') \in \calY^2$,
    \begin{equation*}
        P'_t(y,y') = P'_0(y,y')^{1-t} P'_1(y,y')^{t} \frac{v_{t, \eta}(y')}{\rho_{t, \eta} v_{t, \eta}(y)},
    \end{equation*}
    where $\rho_{t, \eta}$ and $v_{t, \eta}$ are respectively the Perron--Frobenius root and associated eigenvector of $P_0'^{\hadamard (1 - t)} \hadamard P_1'^{\hadamard t}$.
    Observe that by unicity, for any $t \in \bbR$ it holds that $v_{t} = v_{t, 0}, \rho_{t} = \rho_{t, 0}$,
    and
    following analyticity of $\eta \mapsto v_{t, \eta}$ and $\eta \mapsto \rho_{t, \eta}$ on the closed interval $[0, \bar{\eta}/2]$ \citep{kato2013perturbation}, at time parameter $t = t_\star$ achieving \eqref{eq:parameter-time-outside}, it will be convenient to introduce
    \begin{equation*}
        \begin{split}
            \underline{\rho} = \min_{\eta \in [0, \bar{\eta} / 2]} \rho_{t_\star , \eta} > 0, \qquad \underline{v} = \min_{y \in \calY, \eta \in [0, \bar{\eta} / 2]} v_{t_\star , \eta}(y) > 0, \qquad \overline{v} = \max_{y \in \calY, \eta \in [0, \bar{\eta} / 2]} v_{t_\star , \eta}(y).
        \end{split}
    \end{equation*}
    Furthermore, 
    there exists $\bar{\eta}_{\zeta} \in \bbR_+$ such that for $\eta < \bar{\eta}_{\zeta}$
    \begin{equation*}
        \abs{\frac{v_{t_\star, \eta}(y')}{\rho_{t_\star, \eta} v_{t_\star, \eta}(y)} - \frac{v_{t_\star}(y')}{\rho_{t_\star} v_{t_\star}(y)}} \leq \zeta.
    \end{equation*}
    For $(y,y') \in \calE' \setminus \set{ (y_\star, y_\star'), (y_\star, \bar{y}'_\star) }$ and any $t \in \bbR$,
    \begin{equation*}
        P'_t(y,y') - P_t(y,y') = P_0(y,y')^{1-t} P_1(y,y')^{t} \left(\frac{v_{t, \eta}(y')}{\rho_{t, \eta} v_{t, \eta}(y)} - \frac{v_{t}(y')}{\rho_{t} v_{t}(y)}\right),
    \end{equation*}
thus,
\begin{equation*}
        \abs{P'_{t_\star}(y,y') - P_{t_\star}(y,y')} \leq \abs{\frac{v_{t_\star, \eta}(y')}{\rho_{t_\star, \eta} v_{t_\star, \eta}(y)} - \frac{v_{t_\star}(y')}{\rho_{t_\star} v_{t_\star}(y)}}.
    \end{equation*}
    Additionally,
    \begin{equation*}
    \begin{split}
    P'_t(y_\star, y'_\star) - P_t(y_\star, y'_\star) &= P'_t(y_\star, y'_\star) = \eta \frac{v_{t, \eta}(y'_\star)}{\rho_{t, \eta} v_{t, \eta}(y_\star)}, \\
\end{split}
\end{equation*}
leading to
\begin{equation*}
\begin{split}
    \abs{P'_{t_\star}(y_\star, y'_\star) - P_{t_\star}(y_\star, y'_\star)} &\leq \eta \frac{\overline{v}}{\underline{\rho} \underline{v}}.
    \end{split}
    \end{equation*}
    Finally, observing that
    \begin{equation*}
    \begin{split}
    &P'_t(y_\star, \bar{y}'_\star) - P_t(y_\star, \bar{y}'_\star) \\
 =& \left(P_0(y_\star, \bar{y}'_\star) - \eta\right)^{1 - t}\left(P_1(y_\star, \bar{y}'_\star) - \eta\right)^{t} \frac{v_{t, \eta}(\bar{y}'_\star)}{\rho_{t, \eta} v_{t, \eta}(y_\star)} - 
    \left(P_0(y_\star, \bar{y}'_\star) - \eta\right)^{1 - t}\left(P_1(y_\star, \bar{y}'_\star) - \eta\right)^{t}
    \frac{v_{t}(\bar{y}'_\star)}{\rho_{t} v_{t}(y_\star)} \\
&+ \left(P_0(y_\star, \bar{y}'_\star) - \eta\right)^{1 - t}\left(P_1(y_\star, \bar{y}'_\star) - \eta\right)^{t}
    \frac{v_{t}(\bar{y}'_\star)}{\rho_{t} v_{t}(y_\star)} - 
    P_0(y_\star, \bar{y}'_\star)^{1 - t} P_1(y_\star, \bar{y}'_\star)^{t}
    \frac{v_{t}(\bar{y}'_\star)}{\rho_{t} v_{t}(y_\star)} \\
    \end{split}
    \end{equation*}
we obtain that
        \begin{equation*}
    \begin{split}
    \abs{P'_{t_\star}(y_\star, \bar{y}'_\star) - P_{t_\star}(y_\star, \bar{y}'_\star)} \leq& \abs{\frac{v_{t_\star, \eta}(\bar{y}'_\star)}{\rho_{t_\star, \eta} v_{t_\star, \eta}(y_\star)} - 
    \frac{v_{t_\star}(\bar{y}'_\star)}{\rho_{t_\star} v_{t_\star}(y_\star)}} \\
&+ \left(1 - \left(1 - \frac{\eta}{P_0(y_\star, \bar{y}'_\star)}\right)^{1 - t_\star}\left(1 - \frac{\eta}{P_1(y_\star, \bar{y}'_\star)}\right)^{t_\star}
     \right) \frac{\overline{v}}{\underline{\rho} \underline{v}}, \\
    \end{split}
    \end{equation*}
    and for 
    \begin{equation*}
        \begin{split}
            \eta < \min \set{ P_0(y_\star, \bar{y}'_\star) \left(1 - \left(1 - \zeta \frac{\underline{\rho} \underline{v}}{4 \overline{v}} \right)^{\frac{1}{2(1 -t_\star)}}\right), P_1(y_\star, \bar{y}'_\star) \left(1 - \left(1 - \zeta \frac{\underline{\rho} \underline{v}}{4 \overline{v}} \right)^{\frac{1}{2 t_\star}}\right),  \bar{\eta}_{ \zeta/4} },
        \end{split}
    \end{equation*}
    the above is smaller than $\zeta /2$.
    As a result, if we also assume that $\eta$ such that $\eta < \bar{\eta}_{\zeta/2}$ and $\eta < \frac{\zeta \underline{\rho} \underline{v}}{2 \overline{v}}$, it holds that
    \begin{equation*}
        \nrm{P_{t_\star}' - P_{t_\star}} \leq \zeta / 2.
    \end{equation*}

For any $F \in \calF_\kappa(\calY, \calE')$, it then holds from the triangle inequality that
\begin{equation*}
    \nrm{F - P_{t_\star}} \leq \nrm{F - P'_{t_\star}} + \nrm{P'_{t_\star} - P_{t_\star}},
\end{equation*}
thus,
\begin{equation*}
        \nrm{F - P'_{t_\star}} \geq \nrm{F - P_{t_\star}} - \zeta/2 > \zeta/2,
\end{equation*}
and taking the infimum over $\calF_\kappa(\calY, \calE')$ leads to $\gamma^{(e)}_{P'_0, P'_1}(t_\star) \not \in \calF_\kappa(\calY, \calE')$, hence $\gamma^{(e)}_{P'_0, P'_1}(t_\star) \not \in \calW_\kappa(\calY, \calE')$.
Invoking \citet[Corollary~3]{nagaoka2005exponential}, we conclude that $\calW_\kappa(\calY, \calE')$ does not form an e-family.

\paragraph{Block-link case:}

Let us now assume that $\calE'$ and $\calE$ differ by a block-link, that is $\kappa_2(\calE') \setminus \kappa_2(\calE) = \{ (x_\star, x'_\star) \}$,  a single non-merging block, and $\calE' \setminus \calE \subset \calS_{x_\star} \times \calS_{x'_\star}$.
We denote $s = \abs{\calS_{x_\star}}$ and we enumerate $(y_1, y_1'), \dots, (y_s, y_{s}') \in \calS_{x_\star} \times \calS_{x'_\star}$ the elements which are also in $\calE' \setminus \calE$.
Since elements of $\calW_\kappa(\calY, \calE)$ are stochastic matrices, there must exist a block $(x_\star, \bar{x}_\star') \in \kappa_2(\calE), $ with $ \bar{x}_\star' \neq x_\star'$.
We let $\eta \in (0, \bar{\eta}/2]$ where $$\bar{\eta} \eqdef \min \{  P_k(y,y') \colon k \in \{ 0, 1 \}, (y, y') \in  \calS_{x_\star} \times \calS_{\bar{x}'_\star} \}.$$
We further let
$\bar{y}'_1, \dots, \bar{y}'_s \in \calS_{\bar{x}'_\star}$ be such that
$(y_1, \bar{y}_1'), \dots, (y_s, \bar{y}'_{s}) \in \calS_{x_\star} \times \calS_{\bar{x}'_\star}$.
For $k \in \set{0,1}$ and $(y,y') \in \calY^2$, we construct
\begin{equation*}
    P'_{k}(y,y') = \begin{cases} 0 &\text{ when } (y,y') \not \in \calE' \\
        \eta &\text{ when } (y, y' ) \in \calS_{x_\star} \times \calS_{x'_\star}\\
        P_{k}(y,y') - \eta &\text{ when there exists $i \in [s]$ such that $(y,y') = (y_i, \bar{y}_i')$} \\ 
        P_{k}(y,y') &\text{ otherwise.}\\
    \end{cases}
\end{equation*}
Observe that $P'_0$ and $P'_1$ are both irreducible and $\kappa$-lumpable stochastic matrices.
Similar to the edge-link case, we inspect the curve $\gamma^{(e)}_{P_0', P_1'} \colon \bbR \to \calW(\calY, \calE')$, and we denote $P'_t = \gamma^{(e)}_{P_0', P_1'}(t)$ the point on the curve at time parameter $t$.
Using a similar argument as for the edge-link case, we show that the two curves $\gamma^{(e)}_{P_0', P_1'}$ and $\gamma^{(e)}_{P_0, P_1}$ can be made arbitrarily close point-wise, and in particular for $t = t_\star$. That is, there exists $\bar{\eta}_{\zeta}$ such that for $\eta < \bar{\eta}_{ \zeta}$, it holds that $\nrm{P'_{t_\star} - P_{t_\star}} \leq \zeta / 2$.
An application of the triangle inequality and \citet[Corollary~3]{nagaoka2005exponential} conclude the claim.

\paragraph{Chaining and conclusion.}

As a conclusion of the above, if $\calE$ and $\calE'$ only differ by an edge-link or a block-link, the claim holds.
The general case can be obtained by inductively applying the above result to the chain of consecutive edge sets $\{\calE_\ell\}_{\ell = 0 \dots, L}$ with $\calE_0 = \calE, \calE_L = \calE', $ obtained from Lemma~\ref{lemma:chaining-edge-sets}. 

\end{proof}

\begin{definition}[Maximal e-families and minimal non e-families]
We consider the partial order on non-vacuous lumpable and irreducible families induced from their edge sets ordered by inclusion. We call $\calW_\kappa(\calY, \calE)$ a maximal e-family, if $\calW_\kappa(\calY, \calE)$ forms an e-family, and any non-vacuous lumpable family $\calW_{\kappa}(\calY, \calE')$ such that $\calE \subsetneq \calE'$ does not form an e-family.
    Similarly, we call $\calW_\kappa(\calY, \calE')$ a minimal non e-family, if $\calW_\kappa(\calY, \calE')$ does not form an e-family, and any non-vacuous lumpable family $\calW_{\kappa}(\calY, \calE)$ such that $\calE \subsetneq \calE'$ forms an e-family.
\end{definition}

\begin{remark}[Well-definedness]
    The family $\calW_\kappa(\calY, \calY^2)$ defined over the complete graph being lumpable for any $\kappa$, irreducible, and not forming an e-family (in non-trivial settings) it always holds that for any $\calE \subsetneq \calY^2$ such that $\calW_{\kappa}(\calY, \calE)$ forms an e-family, that there exists
    $\calE' \subset \calY^2$ with $\calE \subset \calE'$ such that $\calW_{\kappa}(\calY, \calE')$ does not form an e-family.
    Note however that for any $\calE' \subset \calE$ such that $\calW_\kappa(\calY, \calE')$ is non-vacuous and does not form an e-family, it is not always possible to extract $\calE \subset \calE'$ such that $\calW_{\kappa}(\calY, \calE)$ forms a non-vacuous e-family.
    The example below demonstrates this fact.
\begin{equation*}
    \calW_{\kappa}(\calY, \calE) \sim \left(\begin{tblr}{cc|[dashed]cc}
  \zero & \zero & \zero & \plus  \\
  \zero & \zero & \plus & \zero \\\hline[dashed]
  \zero & \plus & \zero & \plus \\
  \plus & \plus & \zero & \plus \\
\end{tblr}\right)
\end{equation*}
    Observe that if we remove any subset of edges from the family above, it would become either reducible, or non-lumpable. The minimality statement thus only holds vacuously in this case.
\end{remark}

It is instructive to observe that adding or removing diagonal blocks from the diagonal does not alter the e-family property.

\begin{proposition}[Stability through diagonal modification]
\label{proposition:stability-diagonal-modification}
    Let $\calE, \calE'$ where $\calE \subset \calE' \subset \calY^2$ be such that $\calW_{\kappa}(\calY, \calE') \neq \emptyset$, and suppose additionally that
        $$\calE' \setminus \calE \subset \set{(y,y) \colon y \in  \calY, \kappa(y) \not \in \calD = \kappa_2(\calE)}.$$
    Then it holds that $\calW_\kappa(\calY, \calE')$ forms an e-family if and only if $\calW_\kappa(\calY, \calE)$ forms an e-family.
\end{proposition}
\begin{proof}
Let us first assume that $\calE'$ and $\calE$ differ on a single diagonal block $(x_0, x_0) \in \calD'$ and suppose that $\calW_\kappa(\calY, \calE)$ forms an e-family.
We let $P_0', P_1'$ be an arbitrary pair in $\calW_\kappa(\calY, \calE')$ and for $t \in \bbR_+$, we define $\widetilde{P'_t} = {P'_0}^{\hadamard (1 - t)} \hadamard {P'_1}^{\hadamard t}$ to be their log-affine combination.
For $i \in \{0,1\}$, we additionally define $P_i \in \calW_{\kappa}(\calY, \calE)$ be such that
\begin{equation*}
    P_i(y,y') = \begin{cases}
                    P'_i(y,y') &\text{ when }  \kappa(y) \neq x_0 \\
                    0 &\text{ when } (\kappa(y), \kappa(y')) = (x_0, x_0) \\
                    \frac{P'_i(y,y')}{\sum_{x' \neq x_0} P'^{\flat}_i(x_0, x') }  &\text{ otherwise}.
                \end{cases}
\end{equation*}
Since $\calW_\kappa(\calY, \calE)$ forms an e-family, it must be that $\stoch(\widetilde{P}_t) \in \calW_\kappa(\calY, \calE)$ where
$\widetilde{P}_t= P_0^{\hadamard (1-t)} \hadamard P_1^{\hadamard t}$.
Thus, from Corollary~\ref{corollary:geometric-intersection-interpretation}, there exists $v \in \bbR_+$ such that $\diag(v) \widetilde{P}_t \diag(v)^{-1} \in \calF_{\kappa}(\calY, \calE)$.
Observe now that the same vector $v$ also satisfies $\diag(v) \widetilde{P}'_t \diag(v)^{-1} \in \calF_{\kappa}(\calY, \calE')$, thus $\stoch(\widetilde{P}'_t) \in \calW_\kappa(\calY, \calE')$, which forms an e-family. The case where $\calE$ and $\calE'$ differ by more than one diagonal block can be retrieved using chaining (refer to Lemma~\ref{lemma:chaining-edge-sets}). The converse statement follows more generally by monotonicity (refer to Theorem~\ref{theorem:monotonicity}).
\end{proof}

\section{Full classification for small state spaces}
\label{section:classification}
In this section, we provide an exhaustive classification for state spaces of size two and three.

\begin{remark}[Degenerate lumping function]
\label{remark:degenerate-lumping}
    If $\kappa \colon \calY \to \calX$ is such that $\abs{\calX} \in \set{ 
1, \abs{\calY} }$, then $\calW_{\kappa}(\calY, \calE)$ forms an e-family.
\end{remark}

It immediately follows that for $\abs{\calY} = 2$, every lumpable family forms an e-family.
We proceed to enumerate e-families for the three-state space.

\begin{theorem}[Three-state space classification]
\label{theorem:enumeration-three-state-space}

When $\abs{\calY} = 3$,
 and assuming that $\calW_{\kappa}(\calY, \calE) \neq \emptyset$, the following two statements are equivalent.
 \begin{enumerate}
     \item[\enumi] $\calW_{\kappa}(\calY, \calE)$ forms an e-family.
     \item[\enumii] Either $\kappa$ is degenerate or $\calW_{\kappa}(\calY, \calE)$ is equivalent to one of the 12 below-listed families.
 \end{enumerate}
 \begin{center}
\begin{tblr}{cccc}
$\left(\begin{tblr}{c|[dashed]cc}
  \zero & \plus & \plus\\\hline[dashed]
  \plus & \zero & \zero\\
  \plus & \zero & \zero\\
\end{tblr}\right)$
& 
$\left(\begin{tblr}{c|[dashed]cc}
  \plus & \plus & \plus\\\hline[dashed]
  \plus & \zero & \zero\\
  \plus & \zero & \zero\\
\end{tblr}\right)$ 
& 
$\left(\begin{tblr}{c|[dashed]cc}
  \zero & \zero & \plus\\\hline[dashed]
  \plus & \zero & \plus\\
  \plus & \plus & \zero\\
\end{tblr}\right)$
& 
 $\left(\begin{tblr}{c|[dashed]cc}
  \zero & \zero & \plus\\\hline[dashed]
  \plus & \plus & \zero\\
  \plus & \plus & \zero\\
\end{tblr}\right)$
\\
$\left(\begin{tblr}{c|[dashed]cc}
  \plus & \zero & \plus\\\hline[dashed]
  \plus & \plus & \zero\\
  \plus & \plus & \zero\\
\end{tblr}\right)$
&
$\left(\begin{tblr}{c|[dashed]cc}
  \zero & \plus & \plus\\\hline[dashed]
  \plus & \plus & \zero\\
  \plus & \plus & \zero\\
\end{tblr}\right)$
& 
$\left(\begin{tblr}{c|[dashed]cc}
  \plus & \zero & \plus\\\hline[dashed]
  \plus & \zero & \plus\\
  \plus & \plus & \zero\\
\end{tblr}\right)$
& 
$\left(\begin{tblr}{c|[dashed]cc}
  \plus & \plus & \plus\\\hline[dashed]
  \plus & \plus & \zero\\
  \plus & \plus & \zero\\
\end{tblr}\right)$
\\ 
$\left(\begin{tblr}{c|[dashed]cc}
  \zero & \plus & \plus\\\hline[dashed]
  \plus & \zero & \plus\\
  \plus & \plus & \zero\\
\end{tblr}\right)$
& $\left(\begin{tblr}{c|[dashed]cc}
  \plus & \plus & \plus\\\hline[dashed]
  \plus & \zero & \plus\\
  \plus & \plus & \zero\\
\end{tblr}\right)$
&
$\left(\begin{tblr}{c|[dashed]cc}
  \zero & \plus & \plus\\\hline[dashed]
  \plus & \plus & \zero\\
  \plus & \zero & \plus\\
\end{tblr}\right)$
& 
$\left(\begin{tblr}{c|[dashed]cc}
  \plus & \plus & \plus\\\hline[dashed]
  \plus & \plus & \zero\\
  \plus & \zero & \plus\\
\end{tblr}\right)$
\end{tblr}
\end{center}

\end{theorem}

\begin{proof}
    From Remark~\ref{remark:degenerate-lumping}, we only need to consider the case where $\abs{\calX} = 2$.
Being an exponential family is a property common to the entire equivalence class of lumpable families.
After removing empty lumpable families and grouping them into equivalence classes (refer to Section~\ref{section:lumpability} for the definition of equivalence classes), we obtain 26 cases. The 12 families described in \enumii can be shown to forms e-families by applying Corollary~\ref{corollary:no-multi-row-merging-block-is-sufficient}, while the remaining 14 families can be shown to not be e-families using a dimensional argument introduced later in Theorem~\ref{theorem:dimensional-criterion}.
\end{proof}

\begin{remark}
    In the three-state space setting, we observe that $\calW_{\kappa}(\calY, \calE)$ forming an e-family 
    coincides with
    $\calF_{\kappa}^{+}(\calY, \calE)$ being log-affine.
\end{remark}

\section{Algorithmics for moderate state spaces}
\label{section:algorithmics}
In this section, we provide a worst-case time complexity analysis of the criteria developed in Section~\ref{section:sufficient-condition}, Section~\ref{section:necessary-conditions} and Section~\ref{section:dimensional-criterion}. 
Complexity will be typically expressed as a function of the number of vertices or edges in the graph $(\calY, \calE)$ and the lumping map $\kappa$. 
The space $\calY$ is represented by the integers $\set{1, \dots, \abs{\calY}}$.
As is traditional in the literature, we will use the landau notation $\bigO$.
We assume that the graph $(\calY, \calE)$ is represented in the machine using adjacency lists leading to a total structure of size $\bigO(\abs{\calY} + \abs{\calE})$. Time complexity is measured in terms of elementary field operations (addition, multiplication) and memory accesses.
Storing the graph already requires $\bigO(\abs{\calY} + \abs{\calE})$ operations thus we will take this quantity as our lower bound and be mostly concerned with complexity which exceeds this value.
Recall that the lumping map can have two representations---either partitional or functional.
For the partitional representation, we store an array of size $\abs{\calX}$ corresponding to each elements of the lumped space. In each entry $x \in \calX$ of this array we store a list of elements of $\calY$ lumping into $x$.
For the functional representation, we store an array of size $\abs{\calY}$ where at each entry $y \in \calY$ we store $\kappa(y)$. It is not hard to see that one representation can be constructed from the other in $\bigO(\abs{\calY})$. As a result, both representations can be interchangeably considered, and we will henceforth assume that we can both compute $\kappa(y)$ in $\bigO(1)$ and loop over elements of $\calS_x$ in $\abs{\calS_x}$.

\begin{proposition}[Complexity of basic procedures]
We rely on the subroutines below.
\begin{enumerate}[$(i)$]
    \item[\enumi] Determine whether $(\calY, \calE)$ is strongly connected: $\bigO(\abs{\calY} + \abs{\calE})$.
    \item[\enumii] Determine whether $P \in \calW(\calY, \calE)$ is $\kappa$-lumpable: $\bigO(\abs{\calX}\abs{\calY} + \abs{\calE})$.
    \item[\enumiii] Determine whether $\calW_{\kappa}(\calY, \calE)$ is non-vacuous: $\bigO(\abs{\calX}\abs{\calY} + \abs{\calE})$.
    \item[\enumiv] Construct the lumped graph $(\calX, \calD)$ (as adjacency lists) in $\bigO(\abs{\calX}^2 + \abs{\calE})$.
    \item[\enumv] List merging blocks of $(\calY, \calE)$ with respect to $\kappa$: 
$\bigO(\abs{\calX}\abs{\calY} + \abs{\calE})$.
\end{enumerate}
\end{proposition}
\begin{proof}
Simple algorithms yields \enumii, \enumiii, \enumiv, \enumv, while for \enumi,
    listing strongly connected components can be achieved by running 
Tarjan's strongly connected components algorithm \citep{tarjan1972depth}, 
Kosaraju-Sharir's algorithm \citep{sharir1981strong}, or
the path-based strong component algorithm \citep{dijkstra1976discipline}
which all run in $\bigO(\abs{\calY} + \abs{\calE})$.
\end{proof}

\begin{proposition}[Complexity of no multi-row merging block criterion (Corollary~\ref{corollary:no-multi-row-merging-block-is-sufficient})]
    $\bigO(\abs{\calX} \abs{\calY} + \abs{\calE})$.
\end{proposition}

\begin{proposition}[Complexity of dimensional criterion (Theorem~\ref{theorem:dimensional-criterion})]
    $\bigO(\abs{\calE}^{\omega})$ with $\omega \leq 2.371552$.
\end{proposition}

\begin{proof}
    It should be clear that the bottleneck operation is the computation of the rank of the system of matrices.
     Typically, computing the rank of a matrix is done by Gaussian elimination. For a system of $m$ vectors of dimension $n$, this approach can be theoretically implemented in $\bigO(m n ^{\omega - 1})$
\citep{ibarra1982generalization} where $2 \leq \omega \leq 2.371552$ is the matrix multiplication exponent \citep{williams2024new}.
In our case, we have at most $\abs{\calE} + \abs{\calY}$ vectors, each of dimension $\abs{\calE}$. As a result, we obtain a worst-case time complexity of $\bigO((\abs{\calE} + \abs{\calY}) \abs{\calE}^{\omega - 1})$.
\end{proof}

\begin{remark}
In practice---that is for most implementations---the complexity is of order $\bigO(\abs{\calE}^3)$.
    Note that with a parallel algorithm, it is possible to deterministically compute this rank in $\bigO(\log^2\abs{\calY})$ time \citep{mulmuley1986fast} using a polynomial number of processors.
    Since the rank calculation is the bottleneck, distributing this task would substantially improve the efficiency of our algorithm.
\end{remark}

\begin{proposition}
    There exists a $\bigO(\abs{\calY}^\omega)$ time  verifiable witness that can be used to conclude that a lumpable family is not an e-family.
\end{proposition}

\begin{proof}
Suppose that $\calW_{\kappa}(\calY , \calE)$ does not form an e-family. Then there exists a $P_0, P_1 \in \calW_{\kappa}(\calY , \calE)$ and $t \in \bbR$ such that $\gamma_{P_0, P_1}^{(e)}(t)$ is not $\kappa$-lumpable. 
The triplet $(P_0, P_1, t)$ is a witness and it remains to argue that it can be verified in polynomial time.
Since constructing a point at parameter $t$ on the e-geodesic amounts to computing the Perron-Frobenius eigenpair of a Hadamard product of two matrices of size $\abs{\calY} \times \abs{\calY}$, it follows that
$\gamma_{P_0, P_1}^{(e)}(t)$ can be computed in $\bigO(\abs{\calY}^\omega)$, and it can be verified that it is not $\kappa$-lumpable in $\bigO(\abs{\calX} \abs{\calY} + \abs{\calE})$.
\end{proof}

\section*{Acknowledgments}
We thank Hiroshi Nagaoka and Jun'ichi Takeuchi for enlightening discussions and for bringing to our attention an error in a earlier version of this manuscript.

\bibliography{bibliography}
\bibliographystyle{abbrvnat}

\end{document}